\documentclass{article}
\usepackage[english]{babel} 
\usepackage[utf8]{inputenc}
\usepackage{mathtools}
\usepackage{amssymb}
\usepackage{amsthm}

\usepackage{dirtytalk}
\usepackage{xcolor}
\usepackage{framed}
\definecolor{shadecolor}{RGB}{153,204,255}

\newtheorem{theorem}{Theorem}[section]
\newtheorem{definition}[theorem]{Definition}

\newtheorem{lemma}[theorem]{Lemma}
\newtheorem{proposition}[theorem]{Proposition}

\DeclareMathOperator*{\argmax}{arg\,max}
\DeclareMathOperator*{\argmin}{arg\,min}

\theoremstyle{remark}
\newtheorem*{remark}{Remark}
\newtheorem*{conjecture}{Conjecture}

\newtheorem{claimbbb}{Claim}

\newcommand{\rationals}{\ensuremath{{\mathbb{Q}}}}

\newcommand{\defhigh}[1]{\textsc{#1}}
\newcommand{\cupdisjoint}{\overset{\cdot}{\cup}}

\newcommand{\func}[1]{\gamma_{{#1}}} 
\newcommand{\funsymbc}{\gamma}
\newcommand{\fund}[1]{\delta_{{#1}}}
\newcommand{\funsymbd}{\delta}

\newcommand{\ee}[1]{e {{#1}}} 
\newcommand{\eesymb}{e}

\newcommand{\intuunique}[1]{\ensuremath{U_{#1}} }
\newcommand{\intiunique}[1]{\ensuremath{I_{#1}} }

\newcommand{\zindex}[1]{z({#1})}
\newcommand{\zindexfunsymb}{{z}}

\newcommand{\intrational}[2]{\ensuremath{R[{#1},{#2}]} }
\newcommand{\intinit}[1]{\ensuremath{I[{#1},{{#1}}]} }

\newcommand{\intindexone}[2]{\ensuremath{I[{#1},{#2}]} }
\newcommand{\intindextwo}[2]{\ensuremath{J[{#1},{#2}]} }

\newcommand{\intu}[2]{\ensuremath{U_{#1}^{#2}} }

\newcommand{\intv}[2]{\ensuremath{V_{#1}^{#2}} }

\newcommand{\rightendofu}[2]{\ensuremath{\boldsymbol{r}^{#2}({#1})} }

\newcommand{\miller}[1]{\ensuremath{M}(#1) }

\newcommand{\covertest}[2]{\ensuremath{{k_{#1}({#2})}} }
\newcommand{\coverset}[2]{\ensuremath{{\Tilde{k}_{#1}({#2})}} }
\newcommand{\coverall}[2]{\ensuremath{{K_{#1}({#2})}} }
\newcommand{\covertestfun}[1]{\ensuremath{{k_{#1}}} }
\newcommand{\coversetfun}[1]{\ensuremath{{\Tilde{k}_{#1}}} }
\newcommand{\coverallfun}[1]{\ensuremath{{K_{#1}}} }

\newcommand{\redsolovay}[1][\le]{\ensuremath{{#1}_{\mathrm{S}} } }

\newcommand{\redsolovayzweia}[1][\le]{\ensuremath{{#1}_{\mathrm{S}}^{\mathrm{2a}} } }

\begin{document}

\title{Extending the Limit Theorem of Barmpalias and Lewis-Pye to all reals}

\author{Ivan Titov}
\date{\vspace{-5ex}}

\maketitle
\begin{abstract}
By a celebrated result of Ku\v{c}era and Slaman~\cite{Kucera-Slaman-2001}, the Martin-Löf random left-c.e.\  reals form the highest left-c.e.\ Solovay degree. Barmpalias and Lewis-Pye~\cite{Barmpalias-Lewispye-2017} strengthened this result by showing that, for all left-c.e.\ reals $\alpha$ and $\beta$ such that~$\beta$ is Martin-Löf random and all left-c.e.\ approximations $a_0,a_1,\dots$ and $b_0,b_1,\dots$ of $\alpha$ and $\beta$, respectively, the limit
\begin{equation*}
    \lim\limits_{n\to\infty}\frac{\alpha - a_n}{\beta - b_n}
\end{equation*}
exists and does not depend on the choice of the left-c.e.\ approximations to~$\alpha$ and~$\beta$.

Here we give an equivalent formulation of the result of Barmpalias and Lewis-Pye in terms of nondecreasing translation functions and generalize their result to the set of all (i.e., not necessarily left-c.e.) reals.
\end{abstract}

\section{Introduction and background}

\subsection*{Preliminaries}

We assume the reader to be familiar with the basic concepts and results of algorithmic randomness. Our notation is standard. Unexplained notation can be found in Downey and Hirschfeldt~\cite{Downey-Hirschfeldt-2010}. As it is standard in the field, all rational and real numbers are meant to be in the unit interval $[0,1)$, unless stated otherwise.

\medskip
We start with reviewing some central concepts and results that will be used subsequently. The main object of interest of this article is Solovay reducibility, which has been introduced by Robert M.\ Solovay~\cite{Solovay-1975} in 1975 as a measure of relative randomness. Its original definition by Solovay uses the notion of translation function defined on the left cut of a real.

\begin{definition}
\begin{enumerate}
    \item
    A \defhigh{computable approximation} is a computable Cauchy sequence, i.e., a computable sequence of rational numbers that converges.
    A real is \defhigh{computably approximable}, or \defhigh{c.a.}, if it is the limit of some computable approximation.
    \item
    A \defhigh{left-c.e.\ approximation} is a nondecreasing computable approximation.
    A real is \defhigh{left-c.e.}\ if it is the limit of some left-c.e.\ approximation.
\end{enumerate}
\end{definition}

\begin{definition}
    The \defhigh{left cut} of a real $\alpha$, written~$LC(\alpha)$, is the set of all rationals strictly smaller than $\alpha$.
\end{definition}

\begin{definition}[Solovay, 1975]\label{Solovay-reducibility}
    A \defhigh{translation function} from a real~$\beta$ to a real~$\alpha$ is a partially computable function $g$ from the set~$\rationals \cap [0,1)$ to itself such that, for all~$q<\beta$, the value $g(q)$ is defined and fulfills~$g(q)<\alpha$, and
    \begin{equation}\label{eq:translation-function}
        \lim\limits_{q\nearrow\beta} g(q) = \alpha,
    \end{equation}
    where $\lim\limits_{q\nearrow\beta}$ denotes the left limit.
    
    A real~$\alpha$ is \defhigh{Solovay reducible} to a real~$\beta$, also written as~$\alpha\redsolovay\beta$, if there is a real constant~$c$ and a translation function $g$ from~$\beta$ to~$\alpha$ such that, for all~$q< \beta$, it holds that
    \begin{equation}\label{eq:Solovay-reducibility}
        0<\alpha - g(q)<c(\beta - q).
    \end{equation}
\end{definition}
We will refer to~\eqref{eq:Solovay-reducibility} as \defhigh{Solovay condition} and to~$c$ as \defhigh{Solovay constant}, and we say that $g$ \defhigh{witnesses} the Solovay reducibility of $\alpha$ to $\beta$. 

Note that if a partially computable rational-valued function $g$ is defined on all of the set~$LC(\beta)$ and maps it to $LC(\alpha)$, then the Solovay condition~\eqref{eq:Solovay-reducibility} implies~\eqref{eq:translation-function}.

\medskip

Noting that the translation function $g$ defined above provides any useful information only about the left cuts of $\alpha$ and $\beta$, many researchers focused on Solovay reducibility as a measure of relative randomness of left-c.e.\ reals, whereas, outside of the left-c.e.\ reals, the notion has been considered  as \say{badly behaved} by several authors (see e.g. Downey and Hirschfeldt~\cite[Section~9.1]{Downey-Hirschfeldt-2010}).

Calude, Hertling, Khoussainov, and Wang~\cite{Calude-etal-2001} gave an equivalent characterization of Solovay reducibility on the set of the left-c.e.\ reals in terms of left-c.e.\ approximations of the involved reals.

\begin{proposition}[Calude et al., 1998]\label{Solovay-reducibility:index}
    A left-c.e.\ real $\alpha$ is Solovay reducible to a left-c.e.\ real $\beta$ with a Solovay constant $c$ if and only if, for every left-c.e.\ approximations $a_0,a_1,\dots\nearrow\alpha$ and $b_0,b_1,\dots\nearrow\beta$, there exists a computable index function $f:\mathbb{N}\to\mathbb{N}$ such that, for every $n$, it holds that
        \begin{equation}\label{eq:Solovay-reducibility-index-function}
            \alpha - a_{f(n)}<c(\beta - b_n).
        \end{equation}
\end{proposition}

Informally speaking, the reduction $\alpha\redsolovay\beta$ provides for every left-c.e.\ approximation of $\beta$ a not slower left-c.e.\ approximation of $\alpha$. It is easy to see that the universal quantification over left-c.e.\ approximations to~$\alpha$ in Proposition~\ref{Solovay-reducibility:index} can be replaced by an existential quantification as follows.

\begin{proposition}\label{Solovay-reducibility:index:exist}
        A left-c.e.\ real $\alpha$ is Solovay reducible to a left-c.e.\  real  $\beta$ with a Solovay constant $c$ if and only if  there exist left-c.e.\ approximations $a_0,a_1,\dots\nearrow\alpha$ and $b_0,b_1,\dots\nearrow\beta$ such that, for every $n$, it holds that
        \begin{equation}\label{eq:Solovay-reducibility-index}
            \alpha - a_n<c(\beta - b_n).
        \end{equation}
\end{proposition}

In what follows, we refer to the characterizations of Solovay reducibility given in Definition~\ref{Solovay-reducibility} and in Propositions~\ref{Solovay-reducibility:index} and~\ref{Solovay-reducibility:index:exist} as \defhigh{rational} and \defhigh{index} approaches, respectively.

\begin{remark}
For example, Zheng and Rettinger~\cite{Zheng-Rettinger-2004} used the index approach to  introduce {S2a}-reducibility on the c.a.\ reals, a variant of Solovay reducibility, which is equivalent to Solovay reducibility on the left-c.e.\ reals~\cite[Theorem~3.2(2)]{Zheng-Rettinger-2004} but is strictly weaker than Solovay reducibility on the c.a.\ reals~\cite[Theorem~2.1]{Titov-2024-next}.
Some authors~\cite{Hoyrup-etal-2018, Miller-2017, Rettinger-Zheng-2021} use {S2a}-reducibility and not Solovay reducibility as a standard reducibility for investigating the c.a.\ reals.
\end{remark}

\subsection*{The Limit Theorem of Barmpalias and Lewis-Pye on left-c.e.\ reals: two versions}

Using the index approach of Calude et al., Ku\v{c}era and Slaman~\cite{Kucera-Slaman-2001} have proven
that the $\Omega$-like reals, i.e., Martin-Löf random left-c.e.\ reals, form the highest Solovay degree on the set of left-c.e.\ reals.
The core of their proof is the following assertion.

\begin{lemma}[Ku\v{c}era and Slaman, 2001; explicitly: Miller, 2017]\label{KS-LEFT-CE}
For every left-c.e.\ approximations $a_0,a_1,\dots$ and $b_0,b_1,\dots$ of a left-c.e\ real $\alpha$ and a Martin-Löf random left-c.e.\ real $\beta$, respectively,
there exists a constant $c$ such that 
\begin{equation}
    \forall n\in\mathbb{N} \big( \frac{\alpha - a_n}{\beta - b_n} < c \big).
\end{equation}
\end{lemma}
For an explicit proof of the latter lemma, see Miller~\cite[Lemma 1.1]{Miller-2017}.
Barmpalias and Lewis-Pye~\cite{Barmpalias-Lewispye-2017} have strengthened Lemma~\ref{KS-LEFT-CE} by showing the following theorem.

\begin{theorem}[Barmpalias, Lewis-Pye, 2017]
\label{BLP-LEFT-CE-index-form}
    For every left-c.e.\ real $\alpha$ and every Martin-Löf random left-c.e.\ real $\beta$, there exists a constant $d\geq 0$ such that, for every left-c.e.\ approximations $a_0,a_1,\dots\nearrow\alpha$ and $b_0,b_1,\dots\nearrow\beta$, it holds that
    \begin{equation}\label{eq:BLP-LEFT-CE-index-form}
        \lim\limits_{n\to\infty}\frac{\alpha - a_n}{\beta - b_n} = d.
    \end{equation}
    Moreover, $d=0$ if and only if $\alpha$ is not Martin-Löf random.
\end{theorem}

We refer to Theorem~\ref{BLP-LEFT-CE-index-form} as \defhigh{index form of the Limit Theorem of Barmpalias and Lewis-Pye}. We will argue in connection with Proposition~\ref{rational-index-equivalence} below that the index form of the Limit Theorem, which is essentially the original formulation, can be equivalently stated, with the value of~$d$ preserved, in the following rational form. The rational form, however, necessitates the use of nondecreasing translation functions.

\begin{theorem}[Rational form of the Limit Theorem of Barmpalias and Lewis-Pye]
\label{BLP-LEFT-CE-rational-form}
    For every left-c.e.\ real $\alpha$ and every Martin-Löf random left-c.e.\ real $\beta$, there exists a constant $d\geq 0$ such that, for every nondecreasing translation function $g$ from $\beta$ to $\alpha$, it holds that
    \begin{equation}\label{eq:BLP-LEFT-CE-rational-form}
        \lim\limits_{q\nearrow\beta}\frac{\alpha - g(q)}{\beta - q} = d.
    \end{equation}
    Moreover, $d=0$ if and only if $\alpha$ is not Martin-Löf random.
\end{theorem}
The following Proposition~\ref{rational-index-equivalence} and its proof indicate that the Theorems~\ref{BLP-LEFT-CE-rational-form} and~\ref{BLP-LEFT-CE-index-form} can be considered as variants of each other.
\begin{proposition}\label{rational-index-equivalence}
Theorems~\ref{BLP-LEFT-CE-index-form} and~\ref{BLP-LEFT-CE-rational-form} are equivalent while preserving the value of~$d$.
\end{proposition}

\begin{proof}
First, we show that Theorem~\ref{BLP-LEFT-CE-index-form} easily follows from Theorem~\ref{BLP-LEFT-CE-rational-form} while preserving the value of~$d$. Let $a_0,a_1,\dots$ and $b_0,b_1,\dots$ be left-c.e.\ approximations of reals~$\alpha$ and $\beta$ where~$\beta$ is Martin-Löf random. Let~$d$ be as in Theorem~\ref{BLP-LEFT-CE-rational-form}. Define functions~$f$ and~$h$ on~$LC(\beta)$ by
\begin{equation}
    f(q) = \max\{a_0,a_{\max\{t \colon b_t < q\}}\}\makebox[4em]{and} h(q) = a_{\min\{t \colon b_t > q\}}.
\end{equation}
Recall that both the sequences~$a_0, a_1,\dots$ and~$b_0, b_1,\dots$ are nondecreasing (but not necessarily strictly increasing). So, by construction,~$f$ and~$h$ are nondecreasing translation functions from $\beta$ to $\alpha$, and we have for all~$n$ that ${f(b_n) \le a_n \le h(b_n)}$.
As a consequence, we obtain
\begin{equation}\label{eq:an-bn-limit-between-f-and-h-limit}
d = \lim\limits_{n\to\infty}\frac{\alpha - h(b_n)}{\beta - b_n} 
\leq  \liminf \limits_{n\to\infty}\frac{\alpha - a_n}{\beta - b_n} 
\leq  \limsup \limits_{n\to\infty}\frac{\alpha - a_n}{\beta - b_n} 
\leq \lim\limits_{n\to\infty}\frac{\alpha - f(b_n)}{\beta - b_n} = d,
\end{equation}
where the equalities hold by choice of~$d$ and by applying Theorem~\ref{BLP-LEFT-CE-rational-form} to~$h$ and to~$f$.
Now, in particular, the limit inferiore and limit superiore in~\eqref{eq:an-bn-limit-between-f-and-h-limit} are both equal to~$d$, i.e., the corresponding sequence of fractions converges to~$d$.
Theorem~\ref{BLP-LEFT-CE-index-form} follows because $a_0,a_1,\dots$ and $b_0,b_1,\dots$ have been chosen as arbitrary left-c.e.\ approximations of $\alpha$ and $\beta$, respectively.

Next we show that Theorem~\ref{BLP-LEFT-CE-rational-form} easily follows from Theorem~\ref{BLP-LEFT-CE-index-form}. Let~$\alpha$ and~$\beta$ be left-c.e.\ reals where~$\beta$ is Martin-Löf random, and fix some strictly increasing left-c.e.\ approximation $b_0,b_1,\dots$ of~$\beta$. Let~$d$ be as in Theorem~\ref{BLP-LEFT-CE-index-form}. Let~$g$ be an arbitrary nondecreasing translation function from~$\beta$ to~$\alpha$. For all~$n$, let~$a_n=g(b_n)$.
Then, for all rationals~$q$ and for~$n$ such that~$b_n \le q < b_{n+1}$, the monotonicity of~$g$ implies that ${a_n \leq g(q) \leq a_{n+1}}$, hence, for all such~$q$ and~$n$, it holds that
\begin{equation}\label{eq:sandwich-rho-between-rhoone-and-rhotwo}
\underbrace{\frac{\alpha-a_{n+1}}{\beta - b_n}}_{= \rho_1(n)} 
\le \underbrace{\frac{\alpha-g(q)}{\beta - q}}_{= \rho(q)}
< \underbrace{\frac{\alpha-a_n}{\beta - b_{n+1}}}_{= \rho_2(n)}.
\end{equation}
Now, the sequences~$b_0, b_1,\dots$ and~$b_1, b_2,\dots$ are both left-c.e.\ approximations of~$\beta$, and, by choice of~$g$ and by~\eqref{eq:translation-function}, the sequences $a_1, a_2,\dots$ and $a_0, a_1,\dots$ are both left-c.e.\ approximations of~$\alpha$. Hence, by Theorem~\ref{BLP-LEFT-CE-index-form}, we have
\[
d= \lim\limits_{n\rightarrow\infty} \rho_1(n)= \lim\limits_{n\rightarrow\infty} \rho_2(n). 
\]
So for given~$\varepsilon>0$, there is some index~$n(\varepsilon)$ such that, for all~$n> n(\varepsilon)$, the values~$\rho_1(n)$ and~$\rho_2(n)$ differ at most by~$\varepsilon$ from~$d$. But then, by~\eqref{eq:sandwich-rho-between-rhoone-and-rhotwo}, for every rational~$q$ where~$b_{n(\varepsilon)} < q < \beta$, the value~$\rho(q)$ differs at most by~$\varepsilon$ from~$d$. Thus, we have~\eqref{eq:BLP-LEFT-CE-rational-form}, i.e., the values~$\rho(q)$ converge to~$d$ when~$q$ tends to~$\beta$ from the left. Since~$g$ was chosen as an arbitrary nondecreasing translation function from~$\beta$ to~$\alpha$, Theorem~\ref{BLP-LEFT-CE-rational-form} follows.  
\end{proof}

Monotone translation functions have been considered before by Kumabe, Miyabe, Mizusawa, and Suzuki~\cite{Kumabe-etal-2020}, who characterized Solovay reducibility on the set of left-c.e.\ reals in terms of nondecreasing real-valued translation functions.

\medskip

It is not complicated to check that,
in case a left-c.e.\ real is Solovay reducible to another left-c.e.\ real, this can always be
witnessed by some nondecreasing translation function while preserving any given Solovay constant.

\begin{proposition}\label{monotonizable}
Let~$\alpha$ and~$\beta$ be left-c.e.\ reals, and let~$c$ be a real. Then~$\alpha$ is Solovay reducible to~$\beta$ with the Solovay constant~$c$ if and only~$\alpha$ is Solovay reducible to~$\beta$ via a nondecreasing translation function $g$ and the Solovay constant~$c$. 
\end{proposition}

\begin{proof}
For a proof of the nontrivial direction of the asserted equivalence, assume that $\alpha\redsolovay\beta$ with the Solovay constant $c$. By Proposition~\ref{Solovay-reducibility:index}, choose left-c.e.\ approximations $a_0,a_1,\dots\nearrow\alpha$ and $b_0,b_1,\dots\nearrow\beta$ such that~\eqref{eq:Solovay-reducibility-index} holds.

Then $\alpha$ is Solovay reducible to~$\beta$ with the Solovay constant $c$ via the nondecreasing translation function $g$ defined by $g(q)=a_{\max\{n\colon b_n\leq q\}}$.
\end{proof} 

\begin{remark}
    Note that strengthening Definition~\ref{Solovay-reducibility} by considering only nondecreasing translation functions yields a well-defined reducibility $\leq_S^m$ on $\mathbb{R}$, called \defhigh{monotone Solovay reducibility}. The basic properties of $\leq_S^m$ have been investigated by Titov~\cite[Chapter~3]{Titov-2023}. Note that Proposition~\ref{monotonizable} shows that $\leq_S^{m}$ and $\redsolovay$ coincide on the set of left-c.e.\ reals.
\end{remark}

In Theorem \ref{BLP-LEFT-CE-rational-form}, requiring the function $g$ to be nondecreasing is crucial because, for every $\alpha$ and $\beta$ that fulfills the conditions there, we can construct a nonmonotone translation function $g$ such that the left limit in~\eqref{eq:BLP-LEFT-CE-rational-form} does not exist, as we will see in the next proposition.

\begin{proposition}
    Let $\alpha,\beta$ be two left-c.e.\ reals such that $\alpha\redsolovay\beta$ with a Solovay constant $c$.
    Then there exists a translation function $g$ from~$\beta$ to~$\alpha$ such that $\alpha\redsolovay\beta$ with the Solovay constant $c$ via $g$, wherein
    \begin{equation}\label{eq:liminf=0-limsup>0}
         \liminf\limits_{q\nearrow\beta}\frac{\alpha - g(q)}{\beta - q} = 0\makebox[6em]{and}\limsup\limits_{q\nearrow\beta}\frac{\alpha - g(q)}{\beta - q} > 0.
    \end{equation}
\end{proposition}

\begin{proof} By Proposition \ref{Solovay-reducibility:index:exist}, fix left-c.e.\ approximations 
$a_0,a_1\dots\nearrow\alpha$ and $b_0,b_1,\dots\nearrow\beta$ such that \eqref{eq:Solovay-reducibility-index} holds.

The desired translation function $g$ is defined by letting
\begin{align*}
    &g(b_n+\frac{b_{n+1} - b_n}{2^k}) = a_{n+k} && \text{ for all }n,k>0\text{ in case }b_n\neq b_{n+1},
    \\
    &g(b_n - \frac{b_{n}-b_{n-1}}{3^k}) = a_{n+k} - c(b_{n+k} - b_n) && \text{ for all }n,k>0\text{ in case }b_{n-1}\neq b_n,
    \\
    &g(q) = a_{\min\{n: b_n\geq q\}} && \text{ for all other rationals }q<\beta.
\end{align*}
Obviously, $g$ is partially computable and defined on all rationals $< \beta$. 

So, it suffices to show that $g$ satisfies the conditions~\eqref{eq:Solovay-reducibility} and~\eqref{eq:liminf=0-limsup>0} (recall that the Solovay condition~\eqref{eq:Solovay-reducibility} implies the condition~\eqref{eq:translation-function} in the definition of translation functions).

In order to argue that~\eqref{eq:Solovay-reducibility} holds, we consider three cases:
\begin{itemize}
    \item for $q = b_n + \frac{b_{n+1}-b_n}{2^k}$ for some $n,k>0$ where $b_n\neq b_{n+1}$,~\eqref{eq:Solovay-reducibility} is implied by \[\alpha - g(q) = \alpha - a_{n+k}\leq \alpha - a_{n+1}<c(\beta - b_{n+1})<c(\beta - q);\]
    \item for $q = b_n - \frac{b_n - b_{n-1}}{3^k}$ for some $n,k>0$ where $b_{n-1}\neq b_n$,~\eqref{eq:Solovay-reducibility}
    follows from 
    \[\alpha - g(q) = \alpha - a_{n+k} + c(b_{n+k} - b_n) < c(\beta - b_{n+k}) + c(b_{n+k} - b_n) < c(\beta - q);\]
    \item for all other $q$,~\eqref{eq:Solovay-reducibility} is implied by 
    \[\alpha - g(q) = \alpha - a_{\min\{n:b_n\geq q\}} < c(\beta - b_{\min\{n:b_n\geq q\}})\leq c(\beta - q)\]
\end{itemize}
(note that, in each case, the first strict inequality follows from~\eqref{eq:Solovay-reducibility-index}).

Further, the left part of~\eqref{eq:liminf=0-limsup>0} holds since, for every $n$ such that $b_n\neq b_{n+1}$, the real $\alpha$ is an accumulation point of $g(q)|_{[b_n, b_{n+1}]}$ since $g(b_n+\frac{b_{n+1} - b_n}{2^k})\underset{k\to\infty}{\to}\alpha$.

Finally, the right part of~\eqref{eq:liminf=0-limsup>0} holds since, for every $n$ such that $b_{n-1}\neq b_n$, the constant $c$ is an accumulation point of $\frac{\alpha - g(q)}{\beta - q}|_{[b_{n-1},b_n]}$ since
\[\frac{\alpha - g(b_n - \frac{b_n - b_{n-1}}{3^k})}{\beta - (b_n - \frac{b_n - b_{n-1}}{3^k})} = \frac{\alpha - a_{n+k} + c(b_{n+k} - b_n)}{\beta - b_n + \frac{b_n - b_{n-1}}{3^k}} \underset{k\to\infty}{\to}\frac{c(\beta - b_n)}{\beta - b_n} = c.\]
\end{proof}

The latter proposition motivates to consider the Solovay reducibility via only nondecreasing translation functions for the extension of the Limit Theorem of Barmpalias and Lewis-Pye on $\mathbb{R}$.

\section{The theorem}\label{chapter-the-theorem}

\begin{theorem}\label{theorem:BLP-generalized}
    For every real $\alpha$ and every Martin-Löf random real $\beta$, there exists a constant $d\geq 0$ such that, for every nondecreasing translation function $g$ from $\beta$ to $\alpha$, it holds that
    \begin{equation}\label{eq:BLP-generalized}
        \lim\limits_{q\nearrow\beta}\frac{\alpha - g(q)}{\beta - q} = d.
    \end{equation}
\end{theorem}

\begin{proof}

Let $\alpha$ and $\beta$ be two reals where $\beta$ is Martin-Löf random, and let $g$ be a nondecreasing translation function from~$\beta$ to~$\alpha$.

\medskip

The proof is organized as follows.

In Section~\ref{proof:bounded}, we show that $\alpha$ is Solovay reducible to $\beta$ via the translation function $g$, i.e., that the fraction in~\eqref{eq:BLP-generalized} is bounded from above for $q\nearrow\beta$. This fact can be obtained using Claims~\ref{claim:bounded} through~\ref{claim:beta-not-in-test-implies-reducibility}, which we will state in the beginning of Section~\ref{proof:bounded}, subsequent to introducing some notation. Claims~\ref{claim:bounded} and~\ref{claim:beta-not-in-test-implies-reducibility} follow by arguments that are similar to the ones used in connection with the case of left-c.e.\ reals~\cite{Barmpalias-Lewispye-2017, Miller-2017}, and Claim~\ref{claim:overlapping} can be obtained straightforwardly.

Next, in Section~\ref{proof:exists}, we show that the left limit considered in the theorem exists by assuming the opposite, namely, that left limit inferiore of the fraction in~\eqref{eq:BLP-generalized} does not coincide with its left limit superiore (note that, by the previous section, both of them differ from infinity). The contradiction will be obtained rather directly from Claims~\ref{infinite-covering} through~\ref{general-bound-for-coverall}, which we will also state in the beginning of the section. Claims~\ref{infinite-covering} and~\ref{claim:upper-bound-for-single-stage} follow by arguments that are similar to the ones used in connection with the case of left-c.e.\ reals~\cite{Kucera-Slaman-2001, Miller-2017}, whereas the proof of Claim~\ref{general-bound-for-coverall} is rather involved and has no counterpart in the left-c.e.\ case. 

Finally, in Section~\ref{proof:unique}, we show that the left limit considered in the theorem does not depend on the choice of the translation function by assuming the opposite, namely, that there are two translation functions having different left limits of the fraction in~\eqref{eq:BLP-generalized}.

\subsection*{The notation}

In the remainder of this proof and unless explicitly stated otherwise, the term interval refers to a closed subinterval of the real numbers that is bounded by rationals. Lebesgue measure is denoted by~$\mu$, i.e., the Lebesgue measure, or measure, for short, of an interval~$U$ is~${\mu(U)= \max U - \min U}$.

A \defhigh{finite test} is an empty set or a tuple $A=(U_0,\dots,U_m)$ with $m\geq 0$ where the~$U_i$ are not necessarily distinct nonempty intervals. For such a finite test~$A$, its \defhigh{covering function} is
\begin{align*}
\covertestfun{A} \colon [0,1] &\longrightarrow \mathbb{N},\\
x &\longmapsto \#\{i\in\{0,\dots,m\}: x\in U_i\},
\end{align*}
that is,~$\covertestfun{A}(x)$ is the number of intervals in $A$ that contain the real number~$x$. Furthermore, the \defhigh{measure} of~$A$ is~$\mu(A) = \sum_{i\in\{0,\dots,m\}} \mu(U_i)$. 

It is easy to see that the measure of a given finite test~$A$ can be computed by integrating its covering function on the whole domain $[0,1]$, i.e., for every finite test $A$, it holds that
\begin{equation}\label{eq:measure-cover}
\mu(A) =\int\limits _0^1 \covertest{A}{x}dx,  
\end{equation}
as follows by induction on the number of intervals contained in the finite test~$A$.

The induction base holds true because, in case $A=\emptyset$, we obviously have
\[\mu(A) = 0 = \int\limits_0^1 0 dx = \int\limits_0^1 \covertest{A}{x}dx\]
and, in case~$A=(U)$ is a singleton, the function~$\covertestfun{A}$ is just the indicator function of~$U$, while the induction step follows from additivity of the integral operator because the function~$\covertestfun{(U_0, \dots, U_{n+1})}$ is the sum of~$\covertestfun{(U_0, \dots, U_{n})}$ and~$\covertestfun{(U_{n+1})}$.

Observe that by our definition of covering function, the values of the covering functions of the two tests~$([0.2, 0.3],[0.3,0.7])$ and~$([0.2,0.7])$ differ on the argument~$0.3$. Furthermore, for a given finite test and a rational~$q$, by adding intervals of the form~$[q,q]$ the value of the corresponding covering function at~$q$ can be made arbitrarily large without changing the measure of the test.  However, these observations will not be relevant in what follows since they relate only to the value of covering functions at rationals.

\medskip

For all three sections, we fix some effective enumeration~$p_0, p_1, \dots$ without repetition of the domain of~$g$ and, for all natural $n$, define~$Q_n = \{p_0, \dots, p_n\}$.

\medskip

\subsection{The fraction is bounded from above}\label{proof:bounded}

First, we demonstrate that the translation function $g$ witnesses the reducibility $\alpha\redsolovay\beta$, or, equivalently, that
\begin{equation}\label{eq:bounded}
    \exists c \forall q<\beta \big(\frac{\alpha - g(q)}{\beta - q}<c\big).
\end{equation}

For all $n$ and $i$, we will construct a finite test~$T_i^n$ by an essential modification of the construction used by Miller~\cite[Lemma 1.1]{Miller-2017} in the left-c.e.\ case. The construction is effective in the sense that it always terminates and is uniform in $n$ and $i$.

For every $n$ and $i$, let $Y_i^n$ be the union of all intervals lying in the finite test $T_i^n$ (note that $Y_i^n$ can be represented as a disjoint union of finitely many intervals).
For every $i$, let $Y_i$ denote the union of the sets $Y_i^0,Y_i^1,\dots$ .

The property~\eqref{eq:bounded} can be obtained from the following three claims.

\begin{claimbbb}\label{claim:bounded}
    For every $i$ and $n$, it holds that
    \begin{equation}
        \mu(Y_i^n) < 2^{-(i+1)}.
    \end{equation}
\end{claimbbb}

\begin{claimbbb}\label{claim:overlapping}
    For every $i$ and $n$, it holds that 
    \begin{equation}
        Y_i^{n}\subseteq Y_i^{n+1}.
    \end{equation}
\end{claimbbb}

\begin{claimbbb}\label{claim:beta-not-in-test-implies-reducibility}
    For every $i$, the following implication holds:
    \begin{equation}
        \beta\notin Y_i \implies \alpha\redsolovay\beta\text{ via }g\text{ with the Solovay constant }2^{-i}.
    \end{equation}
\end{claimbbb}

From the first two claims we easily obtain that the Lebesgue measure of the set $Y_i$ is also bounded by~$2^{-(i+1)}$ for every $i$.

For all $i$ and $n > 0$, $Y_i^n\setminus Y_i^{n-1}$ is a disjoint union of finitely many intervals, wherein a list of intervals is computable in $i$ and $n$ because the same holds for $Y_i^n$ and $Y_i^{n-1}$. 

Accordingly, the set $Y_i$ is equal to the union of a set $S_i$ of intervals with rational endpoints that is effectively enumerable in $i$ and where the sum of the measures of these intervals is at most $2^{-(i+1)}$. By the two latter properties, the sequence $S_0,S_1,\dots$ is a Martin-Löf test.

The real $\beta$ is Martin-Löf random, hence the test $S_0,S_1,\dots$ should fail on~$\beta$.
Therefore, we can fix an index $i$ such that $\beta$ is not contained in $Y_i$. By Claim~\ref{claim:beta-not-in-test-implies-reducibility}, we obtain that $\alpha\redsolovay\beta$ via $g$ with the Solovay constant $2^{-i}$, which implies~\eqref{eq:bounded} directly by definition of $\redsolovay$.

\medskip

It remains to construct the finite test~$T_i^n$ uniformly in $i$ and $n$ and check that Claims~\ref{claim:bounded} through~\ref{claim:beta-not-in-test-implies-reducibility} are fulfilled.

\subsection*{Outline of the construction and some properties of the finite test $T_i^n$}

Fix~$n,i\geq 0$. Let~$\{q_0 <  \cdots < q_n\}$ be the set $\{p_0,p_1,\dots,p_n\}$ sorted increasingly (remind that $p_0,\dots,p_n$ are the first $n+1$ elements of the fixed effective enumeration of the domain of $g$). 
Due to the technical reasons, let $q_{n+1} = 1$.
We describe the construction of the finite test~$T_i^n$, which is a reworked version of a construction used by Miller~\cite[Lemma 1.1]{Miller-2017} in connection with left-c.e.\ reals. 

For every two indices $k,m$, such that $0\leq k<m\leq n$, define the interval
\begin{equation}\label{eq:define-intindexone}
    \intindexone{k}{m} =
    \begin{cases}
        [q_m, q_k+ \frac{g(q_m) - g(q_k)}{2^{(i+1)}}]&\text{ if }\frac{g(q_m) - g(q_k)}{q_m - q_k} \geq 2^{i+1},
        \\
        \emptyset &\text{ otherwise},
    \end{cases}
\end{equation}
and put the intersection of~$\intindexone{k}{m}$ with the unit interval $[0,1]$ into the test $T_i^n$. Further, due to the technical reason, for all $k,m$ such that $0\leq m\leq k\leq n$, set $\intindexone{k}{m} = \emptyset$.

\begin{claimbbb}\label{only-neighbor-matters}
    For every index~$m$ in the range~$0,\dots,n$, every real $x\in [q_m, q_{m+1})$ and every $k<m$, the following equivalence holds true:
    \begin{equation}\label{eq:only-neighbor-matters}
        \exists l \big(x\in \intindexone{k}{l}\big) \iff x\in \intindexone{k}{m}.
    \end{equation}
\end{claimbbb}

\begin{proof}
The direction \say{$\impliedby$} is straightforward. To prove \say{$\implies$}, fix an index~$l$ and a real $x\in [q_m, q_{m+1}) \cap \intindexone{k}{l}$. Note that, in case $l>m$, it holds that $\min\intindexone{k}{l} = q_l\geq q_{m+1} > x$, hence $x$ cannot lie in $\intindexone{k}{l}$, so we have $l\leq m$.
Therefore, $x\in \intindexone{k}{l}$ implies that 
\begin{equation}\label{eq:neighbors-right-point}
    x\leq \max\intindexone{k}{l} = q_l+\frac{g(q_l) - g(q_k)}{2^{i+1}} \leq q_m + \frac{g(q_m) - g(q_k)}{2^{i+1}} = \max\intindexone{k}{m},
\end{equation}
where the second inequality holds since $q_l\leq q_m$ by $l\leq m$ and $g(q_l)\leq g(q_m)$ by monotonicity of $g$.
From~\eqref{eq:neighbors-right-point} and $x\geq q_m = \min\intindexone{k}{m}$, we obtain that $x\in \intindexone{k}{m}$.
\end{proof}

\begin{claimbbb}
    For every two indices $k$ and $l$ where $0\leq k <l \leq n$, the following implications hold:
    \begin{eqnarray}
        \label{eq:previous-k-annulated}
        \frac{g(q_l) - g(q_k)}{q_l - q_k} \leq 2^{i+1} &\implies& \forall m>l\big(\intindexone{k}{m}\subseteq \intindexone{l}{m}\big),
        \\
        \label{eq:next-l-annulated}
        \frac{g(q_l) - g(q_k)}{q_l - q_k} \geq 2^{i+1} &\implies& \forall m\big(\intindexone{k}{m}\supseteq \intindexone{l}{m}\big).
    \end{eqnarray}
\end{claimbbb}

\begin{proof}
    Fix $k$ and $l$ such that $0\leq k<l\leq n$.

    In order to prove the first implication, assume that $k,l$
    fulfill $\frac{g(q_l) - g(q_k)}{q_l - q_k} \leq 2^{i+1}$, which is equivalent to
    \begin{equation}\label{eq:l-k>inequality}
        q_l - q_k\geq \frac{g(q_l) - g(q_k)}{2^{i+1}},
    \end{equation}
    and fix $m>l$. For every real $x\in \intindexone{k}{m}$, definition of $\intindexone{k}{m}$ implies the inequality $q_m\leq x\leq q_k+ \frac{g(q_m) - g(q_k)}{2^{i+1}}$.
    Inter alia, it means that 
    \begin{equation}\label{eq:q_l<x-and-x-k-inequality}
        q_l<q_m\leq x\makebox[5em]{and}x - q_k \leq \frac{g(q_m) - g(q_k)}{2^{i+1}}.
    \end{equation}
    Hence, we obtain that
    \begin{equation}\label{eq:x-l-inequality}
        x - q_l = (x - q_k) - (q_l - q_k) \leq \frac{g(q_m) - g(q_k)}{2^{i+1}} - \frac{g(q_l) - g(q_k)}{2^{i+1}} = \frac{g(q_m) - g(q_l)}{2^{i+1}},
    \end{equation}
    where the inequality follows from the right part of~\eqref{eq:q_l<x-and-x-k-inequality} and~\eqref{eq:l-k>inequality}.
    The right side of~\eqref{eq:previous-k-annulated} is implied by the left part of~\eqref{eq:q_l<x-and-x-k-inequality} and~\eqref{eq:x-l-inequality}.

    For the second implication, assume that $k,l$
    fulfill $\frac{g(q_l) - g(q_k)}{q_l - q_k} \geq 2^{i+1}$, which is equivalent to
    \begin{equation}\label{eq:l-k<inequality}
        q_l - q_k\leq \frac{g(q_l) - g(q_k)}{2^{i+1}}.
    \end{equation}
    In case $m\leq l$, the right side of~\eqref{eq:next-l-annulated} is obvious since~$\intindexone{l}{m}=\emptyset$, so it suffices to consider $m>l$. For every real~$x\in\intindexone{l}{m}$, definition of~$\intindexone{l}{m}$ implies the inequality~$q_m\leq x\leq q_l+\frac{g(q_m) - g(q_l)}{2^{i+1}}$. Inter alia, it means that
    \begin{equation}\label{eq:q_k<x}
        q_k<q_l<q_m\leq x\makebox[5em]{and}x-q_l\leq \frac{g(q_m) - g(q_l)}{2^{i+1}}.
    \end{equation}
    Hence, we obtain similar as in the proof of previous implication that
    \begin{equation}\label{eq:x-k-+-inequality}
    x - q_k = (x-q_l)+(q_l-q_k) \leq \frac{g(q_m) - g(q_l)}{2^{i+1}} + \frac{g(q_l)-g(q_k)}{2^{i+1}} = \frac{g(q_m)-g(q_l)}{2^{i+1}}.
    \end{equation}
    The right part of~\eqref{eq:next-l-annulated} is implied by the left part of~\eqref{eq:q_k<x} and~\eqref{eq:x-k-+-inequality}.
\end{proof}

\subsection*{Preliminaries for the proof of Claim~\ref{claim:bounded}}
Let $0=i_0<i_1<\dots<i_s$ be the indices in the range $0,\dots,n$ such that
\begin{equation}\label{eq:first-layer-of-stair-last-step}
    \frac{g(q_m) - g(q_{i_s})}{q_m - q_{i_s}} > 2^{i+1} \makebox[4em]{for all}m\in\{i_s,\dots,n\}
\end{equation}
and, for every $j\in\{0,\dots,s-1\}$,
\begin{eqnarray}
    \label{eq:first-layer-of-stair-left-part}
    \frac{g(q_{i_{j+1}}) - g(q_{i_j})}{q_{i_{j+1}} - q_{i_j}} &\leq& 2^{i+1},
    \\
    \label{eq:first-layer-of-stair-right-part}
    \frac{g(q_m) - g(q_{i_j})}{q_m - q_{i_j}} &>& 2^{i+1} \makebox[4em]{for all}m\in\{i_j+1,\dots,i_{j+1}-1\}.
\end{eqnarray}

Further, due to the technical reasons, we fix an additional index $i_{s+1} = n+1$ (hence $i_{s+1}-1 = n$) and set $q_{n+1} = 1$ and $g(q_{n+1}) = 1$, so~\eqref{eq:first-layer-of-stair-last-step} is nothing but~\eqref{eq:first-layer-of-stair-right-part} for~$j = s$.

\begin{claimbbb}\label{claim-about-dominating-interval}
    For every $j\in\{0,\dots,s\}$ the following property holds true:
\end{claimbbb}
\begin{equation}\label{eq:dominating-interval}
    \forall k< i_{j+1} \bigg(\intindexone{k}{i_j} = \emptyset \makebox[3em]{and}\forall m>i_j \big( \intindexone{k}{m} \subseteq \intindexone{i_j}{m}\big)\bigg).
\end{equation}

\begin{proof}
We fix an $j\in\{0,\dots,s\}$ and proof the claim statement by case distinction for $k$.
\begin{itemize}
    \item
    In case $k = i_h$ for some $h<j$, it holds that
    \begin{equation}\label{eq:inequality-for-previous-k}
        \frac{g(q_{i_j})-g(q_{i_h})}{q_{i_j}-q_{i_h}} = \frac{\big(g(q_{i_j}) - g(q_{i_{j-1}})\big) + \dots + \big(g(q_{i_{h+1}}) - g(q_{i_h})\big)}{(q_{i_j} - q_{i_{j-1}}) + \dots + (q_{i_{h+1}} - q_{i_h})}\leq 2^{i+1},
    \end{equation}
    where the inequality follows from~\eqref{eq:first-layer-of-stair-left-part} applied for indices $h,\dots,j-1$ since, for all natural $l$ and all reals~${A_1,\dots,A_l\geq 0}$ and~${B_1,\dots,B_l,C>0}$, the~$l$ equalities ${\frac{A_1}{B_1}\leq C,\dots,\frac{A_l}{B_l}\leq C}$ imply together that ${\frac{A_1+\dots+A_l}{B_1+\dots+B_l}\leq C}$.
    
    Therefore, we obtain by~\eqref{eq:define-intindexone} that $\intindexone{i_h}{i_j}=\emptyset$,
    and~\eqref{eq:previous-k-annulated} implies for every $m>i_j$ that 
    \begin{equation}\label{eq:i_h-annulated}
        \intindexone{i_h}{m}\subseteq \intindexone{i_j}{m}.
    \end{equation}

    \item
    In case $k\in \{i_h+1,\dots,i_{h+1}-1\}$ for some $h<j$, it holds by choice of $i_h$ that
    \begin{equation}\label{eq:inequality-for-previous-block}
        \frac{g(q_k)-g(q_{i_h})}{q_k-q_{i_h}}> 2^{i+1}.
    \end{equation}
    
    First, we obtain that~$\intindexone{i_h}{i_j}=\emptyset$ by~\eqref{eq:define-intindexone} since
    \begin{equation}\label{eq:i_h+-annulated}
        \frac{g(q_{i_j})-g(q_k)}{q_{i_j}-q_k} = \frac{\big(g(q_{i_j})-g(q_{i_h})\big)-\big(g(q_k) - g(q_{i_h})\big)}{(q_{i_j}-q_{i_h})-(q_k - q_{i_h})} \leq 2^{i+1},
    \end{equation}
    where the inequality follows from~\eqref{eq:inequality-for-previous-k} and~\eqref{eq:inequality-for-previous-block} because, for all reals ${A_1,A_2\geq 0}$ and ${B_1,B_2,C>0}$, the two equalities ${\frac{A_1}{B_1}\leq C}$ and ${\frac{A_2}{B_2}>C}$ imply together that ${\frac{A_1-A_2}{B_1-B_2}\leq C}$.
    
    Second, we obtain for every $m>i_j$ that 
    \begin{equation}
        \intindexone{k}{m}\subseteq \intindexone{i_h}{m}\subseteq \intindexone{i_j}{m},
    \end{equation}
    where the left side is implied by~\eqref{eq:next-l-annulated} due to~\eqref{eq:i_h+-annulated}, and the right side holds by~\eqref{eq:i_h-annulated}.

    \item
    In case $k\in \{i_j,\dots,i_{j+1}-1\}$, we straightforwardly obtain from~$k\geq i_j$ that ${\intindexone{k}{i_j} = \emptyset}$.
    For $k = i_j$, the right side of~\eqref{claim-about-dominating-interval} is trivial; for $k>i_j$, we obtain from the choice of $i_j$ that
    \begin{equation}\label{eq:inequality-for-current-block}
        \frac{g(q_k)-g(q_{i_j})}{q_k-q_{i_j}}> 2^{i+1},
    \end{equation}
    and thus~\eqref{eq:next-l-annulated} implies for every~$m>i_j$ that
    \begin{equation}
        \intindexone{k}{m}\subseteq \intindexone{i_j}{m}.
    \end{equation}
\end{itemize} 
\end{proof}

\begin{claimbbb}\label{non-covered-point}
For every $j\in\{0,\dots,s\}$ and every real ${x\in (q_{i_j} + \frac{g(q_{i_{j+1}}) - g(q_{i_j})}{2^{i+1}}, q_{i_{j+1}})}$, it holds that $x\notin \intindexone{k}{l}$ for all $k$ and $l$ in the range $0,\dots,n$.
\end{claimbbb}
\begin{remark}
    Note that, in case~$j<s$, it holds by~\eqref{eq:first-layer-of-stair-left-part} that ${q_{i_j} + \frac{g(q_{i_{j+1}}) - g(q_{i_j})}{2^{i+1}}\leq q_{i_{j+1}}}$, hence the interval ${(q_{i_j} + \frac{g(q_{i_{j+1}}) - g(q_{i_j})}{2^{i+1}}, q_{i_{j+1}})}$ used in the Claim~\ref{non-covered-point} is well-defined.

    In case~$j = s$, it may occur that~$q_{i_j} + \frac{g(q_{i_{j+1}}) - g(q_{i_j})}{2^{i+1}}> q_{i_{j+1}}$, so, for every two reals~$a>b$, let~$[a,b]$ conventionally denote an empty set.
\end{remark}

\begin{proof}
Fix $j\in\{0,\dots,s\}$ and a real ${x\in (q_{i_j} + \frac{g(q_{i_{j+1}}) - g(q_{i_j})}{2^{i+1}}, q_{i_{j+1}})}$. We accomplish the proof in four consequent steps.

\begin{enumerate}
    \item\label{step-one}
    First, we note that
    $x\in (q_{i_{j+1}-1},q_{i_{j+1}})$ because, in case $i_{j+1}>i_j+1$, the inequality~\eqref{eq:first-layer-of-stair-right-part} for $m=i_{j+1}-1$ implies that
    \[q_{i_{j+1}-1} < q_{i_j} + \frac{g(q_{i_{j+1}-1}) - g(q_{i_j})}{2^{i+1}} < x.\]
    In case $i_{j+1}=i_j+1$, we obviously have $x\in(q_{i_j},q_{i_{j+1}})=(q_{i_{j+1}-1},q_{i_{j+1}})$.
    
    \item\label{step-two}
    Next, we note that ${x\notin \intindexone{i_j}{i_{j+1}-1}}$ since, in case $i_{j+1} = i_j+1$, we have ${\intindexone{i_j}{i_{j+1}-1}=\intindexone{i_j}{i_j}=\emptyset}$, and otherwise,
    \[x\in (q_{i_j} + \frac{g(q_{i_{j+1}}) - g(q_{i_j})}{2^{i+1}}, q_{i_{j+1}}) = (q_{i_j},q_{i_{j+1}})\setminus \intindexone{i_j}{i_{j+1}-1}.\]

    \item\label{step-three}
    Further, we show that $x\notin \intindexone{k}{i_{j+1}-1}$ for all $k$.
    For all ${k\geq i_{j+1}-1}$, this is obvious since ${\intindexone{k}{i_{j+1}-1}=\emptyset}$ by definition, so, in what follows, assume that ${k<i_{j+1}-1}$.
    
    In case ${i_{j+1} = i_j+1}$, the left side of~\eqref{eq:dominating-interval} yields ${\intindexone{k}{i_{j+1} - 1} = \intindexone{k}{i_j} = \emptyset}$.
    In case ${i_{j+1}>i_j+1}$, the converse would imply by the right side of~\eqref{eq:dominating-interval} for ${m = i_{j+1}-1>i_j}$ that ${x\in \intindexone{k}{i_{j+1}-1}\subseteq \intindexone{i_j}{i_{j+1}-1}}$, but that is impossible by Step~\ref{step-two}.
    
    \item\label{step-four}
    Finally, we show that $x\notin \intindexone{k}{l}$ for every $k$ and $l$ by contradiction.
    Suppose that $x\in \intindexone{k}{l}$ for some $k$ and $l$. Since $x\in(q_{i_{j+1}-1},q_{i_{j+1}})$ by Step~\ref{step-one}, we can apply Claim~\ref{only-neighbor-matters} for~$x$ and obtain that ${x\in \intindexone{k}{l}\subseteq \intindexone{k}{i_{j+1}-1}}$, but this is impossible by Step~\ref{step-three}.
\end{enumerate}
\end{proof}

\subsection*{The proof of Claim~\ref{claim:bounded}}

Then, due to 
\[Y_i^n = \bigg( \bigcup_{k,l\in\{0,\dots,n\}}\intindexone{k}{l} \bigg) \cap [0,1],\]
Claim~\ref{non-covered-point}
implies that 
\[(q_{i_j} + \frac{g(q_{i_{j+1}}) - g(q_{i_j})}{2^{i+1}}, q_{i_{j+1}})\cap Y_i^n = \emptyset \makebox[7em]{for every}j\in\{0,\dots,s\}.\]

Hence we obtain that
\begin{equation}\label{eq:step-of-stair-bounded}
    Y_i^n\cap (q_{i_j},q_{i_{j+1}}) \subseteq [q_{i_j},q_{i_j} + \frac{g(q_{i_{j+1}})-g(q_{i_j})}{2^{i+1}}]
\end{equation}
with the measure
\begin{equation}\label{eq:measure-of-step-of-stair}
    \mu\big(Y_i^n\cap (q_{i_j},q_{i_{j+1}})\big) \leq \frac{g(q_{i_{j+1}})-g(q_{i_j})}{2^{i+1}}.
\end{equation}

Therefore, by 
\[Y_i^n\setminus\{q_{i_0},q_{i_1},\dots,q_{i_{s+1}}\} = \big(Y_i^n\cap(q_{i_0},q_{i_1})\big)\cupdisjoint\dots\cupdisjoint\big(Y_i^n\cap(q_{i_s},q_{i_{s+1}})\big),\]
where $A\cupdisjoint B$ denotes the union of two disjoint intervals $A$ and $B$, we obtain an upper bound for the measure of $Y_i^n$:
\begin{equation*}
    \mu(Y_i^n) = \sum_{j=0}^s \mu\big(Y_i^n\cup (q_{i_j},q_{i_{j+1}})\big)
    \leq
    \sum_{j=0}^s \frac{g(q_{i_{j+1}})-g(q_{i_j})}{2^{i+1}} = \frac{g(q_{i_{s+1}}) - g(q_{i_0})}{2^{i+1}} \leq \frac{1}{2^{i+1}}.
\end{equation*}
Here, the first inequality follows from by~\eqref{eq:measure-of-step-of-stair} applied for all $j$ from $0$ to $s$, and the second one is implied by $g(q_0)\geq 0$ and $g(q_{i_{s+1}}) = g(q_{n+1}) = 1$.

\subsection*{The proof of Claim~\ref{claim:overlapping}}
Let $n,i\geq 0$.

The finite test $T_i^n$ is a subset of the finite test $T_i^{n+1}$ since every intersection of an interval $\intindexone{k}{m}$ where $0\leq k < m\leq n$ with~$[0,1)$ added into the test $T_i^n$ will also be added into the test $T_i^{n+1}$ as well. Hence we directly obtain that
\[Y_i^n = \bigcup\limits_{I\in T_i^n} I \subseteq \bigcup\limits_{I\in T_i^{n+1}} I = Y_i^{n+1}.\]

\subsection*{The proof of Claim~\ref{claim:beta-not-in-test-implies-reducibility}}

Fix an index $i$ such that $\beta \notin Y_i$. By Claim~\ref{claim:overlapping}, it means inter alia that $\beta\notin Y_i^n$ for every natural $n$. 

We aim to show that $\alpha\redsolovay\beta$ via $g$ with the Solovay constant $c = 2^{-i}$ by contradiction: fixing a rational $q\in LC(\beta)$ such that
\begin{equation}\label{eq:reducibility-with-2-i+1}
    \alpha - g(q) \geq 2^{-i}(\beta - q) > 2^{-(i+1)}(\beta - q),
\end{equation}
we can, by $\mathrm{dom}(g)\supseteq LC(\beta)$, fix an index $K$ such that $q = p_K$.
We know by definition of translation function that ${\lim\limits_{p\nearrow\beta}\big(g(p) - g(p_K)\big) = \alpha - g(p_K)}$, hence there exists $\epsilon>0$ such that
\begin{equation}\label{eq:p_K-is-enough}
    g(p) - g(p_K) > 2^{-(i+1)}(\beta - p_K)\makebox[5em]{for all}p\in (\beta - \epsilon, \beta).
\end{equation}

Fix an index $M>K$ such that $p_M\in (\beta-\epsilon,\beta)$. Note that~\eqref{eq:p_K-is-enough} implies in particular that $g(p_M)-g(p_K)>0$, hence $p_K<p_M$ because the function $g$ is nondecreasing.

Let~$\{q_0 <  \cdots < q_M\}$ be the set $\{p_0,\dots,p_M\}$ sorted increasingly, and let ${k,m\in\{0,\dots,M\}}$ denote two indices such that $q_k = p_K$ and $q_m = p_M$.
In particular, we have ${q_k=p_K<p_M=q_m}$, hence ${k<m}$.

To obtain a contradiction with $\beta\notin Y_i^M$ and conclude the proof of Claim~\ref{claim:beta-not-in-test-implies-reducibility}, and thus also of~\eqref{eq:bounded}, it suffices to show that $\beta$ lies within one of the intervals of the finite test $T_i^M$, namely, in ${\intindexone{k}{m}\cap [0,1)}$.

Indeed,~$\beta\in[0,1)$ holds obviously, and~$\beta\in \intindexone{k}{m}$ holds by~\eqref{eq:define-intindexone} since
\begin{equation}
    q_m<\beta<q_k+\frac{g(q_m) - g(q_k)}{2^{i+1}},
\end{equation}
where the right inequality is implied by~\eqref{eq:p_K-is-enough} for $p=q_m$.

\medskip

\subsection{The left limit exists}\label{proof:exists}

In this section, we show that, for $q$ converging to $\beta$ from below, the fraction~$\frac{\alpha - g(q)}{\beta - q}$ converges, i.e., that
\begin{equation}\label{eq:blp-generalized}
\exists\lim_{q\nearrow \beta}\frac{\alpha - g(q)}{\beta - q},
\end{equation}
by contradiction.
For all~$q\in LC(\beta)$, the fraction $\frac{\alpha - g(q)}{\beta - g(q)}$ is obviously positive and, by the previous section, bounded, consequently, supposing that the left limit in~\eqref{eq:blp-generalized} does not exist, we can fix two rational constants~$c$ and~$d$ where
\begin{equation}\label{eq:the-choice-of-c-and-d}
c<d, \quad d-c <1, \makebox[4em]{ and } \liminf\limits_{q\nearrow \beta}\frac{\alpha - g(q)}{\beta - q}<c<d<\limsup\limits_{q\nearrow \beta}\frac{\alpha - g(q)}{\beta - q}
\end{equation}
and the rational
\begin{equation}
    e = d-c>0.
\end{equation}

For a given finite subset~$Q$ of the domain of~$g$, we will construct a finite test~$\miller{Q}$ by an extension of a construction used by Miller~\cite[Lemma 1.2]{Miller-2017} in the left-c.e.\ case. The construction is effective in the sense that it always terminates and yields the test~$\miller{Q}$ in case it is applied to a finite subset of the domain of~$g$. 

Further, for every finite subset~$Q$ of the domain of~$g$ and every rational~$p$, we let
\[
\coverset{Q}{p} = \covertest{\miller{Q}}{p} \makebox[6em]{ and } \coverall{Q}{p} = \max_{H \subseteq Q}\coverset{H}{p}.
\] 

The desired contradiction can be obtained from the following three claims.

\begin{claimbbb}\label{infinite-covering}
Let~$Q_0 \subseteq  Q_1 \subseteq \cdots$ be a sequence of finite sets that converges to the domain of~$g$. Then it holds that
\[
\lim\limits_{n\to\infty}\coverall{Q_n}{\ee{\beta}} = \infty.
\]
\end{claimbbb}

\begin{claimbbb}\label{claim:upper-bound-for-single-stage}
For every finite subset~$Q$ of the domain of~$g$, it holds that
\begin{equation}\label{eq:upper-bound-for-single-stage}
\int\limits _0^1 \coverset{Q}{x}dx 
= \mu\big(\miller{Q}\big)  
\leq g(\max Q) - g(\min Q).
\end{equation}
\end{claimbbb}

\begin{claimbbb}\label{general-bound-for-coverall}
For every finite subset~$Q$ of the domain of~$g$ and for every nonrational real $p$ in~$[0,\eesymb]$, it holds that 
\begin{equation}\label{eq:general-bound-for-coverall}
\coverall{Q}{p}\leq \coverset{Q}{p}+1.
\end{equation}
\end{claimbbb}

Remind that~$p_0, p_1, \dots$ is an effective enumeration without repetition of the domain of~$g$ and~$Q_n = \{p_0, \dots, p_n\}$ for~$n=0,1, \dots$ .
We consider a special type of step function with domain~$[0,1]$ that is given by a partition of the unit interval into finitely many intervals with rational endpoints such that the function is constant on the corresponding open intervals but may have arbitrary values at the endpoints. For the scope of this proof, a designated interval of such a step function is an interval that is the closure of a maximum contiguous open interval on which the function attains the same value. I.e., the designated intervals form a partition of the unit interval except that two designated intervals may share an endpoint. Observe that, for every finite subset~$H$ of the domain of~$g$, the corresponding cover function~$\coverset{H}{\cdot}$ is such a step function with values in the natural numbers, and the same holds for the function~$\coverallfun{Q_n}$ since~$Q_n$ has only finitely many subsets. Furthermore, for given~$n$, the designated intervals of the function~$\coverall{Q_n}{\cdot}$ together with the endpoints and function value of every interval are given uniformly effective in~$n$ because~$g$ is computable, and the construction of~$M(Q_n)$ is uniformly effective in~$n$. 

For all natural numbers~$i$ and~$n$, consider the step function~$\coverallfun{Q_n}$ and its designated intervals. For every such interval, call its intersection with~$[0, \eesymb]$ its restricted interval. Let~$X^n_i$ be the union of all restricted designated intervals where on the corresponding designated interval the function~$\coverallfun{Q_n}$  attains a value that is strictly larger than~$2^{i+2}$. Let~$X_i$ be the union of the sets~$X^0_i,  X^1_i,  \dots$ . 

By our assumption that the values of~$g$ are in~$[0,1)$ and by~\eqref{eq:upper-bound-for-single-stage}, for all~$n$, the integral of~$\coverset{Q_n}{p}$ from~$0$ to~$1$ is at most~$1$, hence by~\eqref{eq:general-bound-for-coverall}, the integral of~$\coverall{Q_n}{p}$ from~$0$ to~$\eesymb$ is at most~$2$. Consequently, each set~$X^n_i$ has Lebesgue measure of at most~$2^{-(i+1)}$. The latter upper bound then also holds for the Lebesgue measure of the set~$X_i$ for every $i$ since, by the maximization in the definition of~$\coverallfun{Q_n}$ and 
\[
Q_0 \subseteq Q_1 \subseteq \cdots ,
\makebox[5em]{ we have }
\coverallfun{Q_0} < \coverallfun{Q_1} < \cdots \;,  
\makebox[5em]{ hence }
X^0_i \subseteq X^1_i \subseteq \cdots .
\]
By construction, for all~$i$ and~$n>0$, the difference~$X^n_i \setminus X^{n-1}_i$ is equal to the union of  finitely many intervals that are mutually disjoint except possibly for their endpoints, and a list of these intervals is uniformly computable in~$i$ and~$n$ since the functions~$\coverallfun{Q_n}$ are uniformly computable in~$n$.  Accordingly, the set~$X_i$ is equal to the union of a set~$U_i$ of intervals with rational endpoints that is effectively enumerable in~$i$ and where the sum of the measures of these intervals is at most~$2^{-(i+1)}$. By the two latter properties, the sequence~$U_0, U_1, \dots$ is a Martin-L\"{o}f test. By Claim~\ref{infinite-covering}, the values~$\coverall{Q_n}{\ee{\beta}}$ tend to infinity where~$\ee{\beta}< \eesymb$, hence for all~$n$, the Martin-L\"{o}f random real~$\ee{\beta}$  is contained in some interval in~$U_n$, a contradiction. This concludes the proof that Claim~\ref{infinite-covering} through~\ref{general-bound-for-coverall} together imply that the left limit~\eqref{eq:blp-generalized} exists.

\medskip

It remains to construct the finite Test $\miller{Q}$ for a given finite subset $Q$ of the domain of $g$ and check that Claims~\ref{infinite-covering} through~\ref{general-bound-for-coverall} are fulfilled.

\subsection*{The intervals that are used}
First, we define two partial computable functions~$\funsymbc$ and~$\funsymbd$ that have the same domain as~$g$:
\[
\funsymbc(q)=g(q) - cq \makebox[7em]{ and }  \funsymbd(q)=g(q) - dq.
\]
Due to of~$c < d$, the following claim is immediate.

\begin{claimbbb}\label{claim:properties-gamma-delta}
Whenever~$g(q)$ is defined, we have
\[
\funsymbc(q) - \funsymbd(q) = (d-c)q = \ee{q} >0, \makebox[7em]{ hence } \funsymbc(q) >  \funsymbd(q).
\]
In particular, the partial function~$\funsymbc - \funsymbd$ is strictly increasing on its domain, hence, for every sequence~$q_0 < q_1 < \dots$ of rationals on $[0,\beta)$ that converges to~$\beta$, the values~$g(q_i)$ are defined, and therefore, the values $\funsymbc(q_i) - \funsymbd(q_i)$ converge strictly increasingly to~$(d-c)\beta$.
\end{claimbbb}

Now, for given rationals~$p$ and~$q$, we define the interval
\[
\intrational{p}{q}
= [\funsymbc(p)-\funsymbd(p),\funsymbc(q)-\funsymbd(p)].
\]
From this definition and the definitions of~$\gamma$ and~$\delta$, the following claim is immediate. Note that assertion~(iii) in the claim relates to expanding an interval at the right endpoint.
\begin{claimbbb}\label{claim:properties-of-intervals}
\begin{itemize}
\item[(i)]
Any interval of the form~$\intrational{p}{q}$ has the left endpoint~$\ee{p}$.
\item[(ii)]
Consider an interval of the form~$\intrational{p}{q}$. In case~$\funsymbc(p) \le \funsymbc(q)$, the interval has length~$\funsymbc(q)- \funsymbc(p)$, otherwise, the interval is empty. In particular, any interval of the form~$\intrational{p}{p}$ has length~$0$. 
\item[(iii)]
Let~$\intrational{p}{q}$ be a nonempty interval, and assume~$\funsymbc(q) \le \funsymbc(q^{\prime})$. Then the interval~$\intrational{p}{q}$ is a subset of the interval~$\intrational{p}{q^{\prime}}$, both intervals have the same left endpoint $\ee{p}$, and they differ in length by~$\funsymbc(q^{\prime}) - \funsymbc(q)$.
\end{itemize}
\end{claimbbb}

By the choice~\eqref{eq:the-choice-of-c-and-d} of $c$ and $d$, the real~$\beta$ is an accumulation point of both the sets
\begin{eqnarray*}
S &=& \{q < \beta \colon \frac{\alpha - g(q)}{\beta - q} > d \} = \{q < \beta \colon 
\funsymbd(q) < \alpha - d\beta \},\\
T &=& \{q < \beta \colon \frac{\alpha - g(q)}{\beta - q} < c \} \; =  \{q < \beta \colon 
\funsymbc(q) > \alpha - c\beta \}.
\end{eqnarray*}

The two following claims, which have already been used in the left-c.e.\ case~\cite{Barmpalias-Lewispye-2017, Miller-2017}, will be crucial in the proof of Claim~\ref{infinite-covering}.

\begin{claimbbb}\label{claim:set-s-and-t-are-disjoint}
The sets~$S$ and~$T$ are disjoint.
\end{claimbbb}
\begin{proof}
The claim holds because for every~$q < \beta$, the bounds in the definitions of~$S$ and~$T$ are strictly farther apart than the values~$\funsymbc(q)$ and~$\funsymbd(q)$, i.e., we have
\[
\funsymbc(q) - \funsymbd(q) = (d-c) q < (d-c) \beta = (\alpha - c\beta) - (\alpha - d\beta).
\] 
\end{proof}

\begin{claimbbb}\label{claim:intervals-with-endpoints-in-s-and-t}
Let~$q$ be in~$S$, and let~$q^{\prime}$ be in~$T$. Then the  interval~$\intrational{q}{q^{\prime}}$ contains~$\ee{\beta}$.
\end{claimbbb}
\begin{proof}
By definition, the interval~$\intrational{q}{q^{\prime}}$ has the left endpoint~$\ee{q}$ and the right endpoint~$\funsymbc(q^{\prime})-\funsymbd(q)$.  By definition of the sets~$S$ and~$T$, on the one hand, we have~$q < \beta$, hence~$\ee{q}<\ee{\beta}$, on the other hand, we have
\[
\funsymbc(q^{\prime})-\funsymbd(q) > (\alpha - c\beta) - (\alpha - d\beta) = (d-c)\beta = \ee{p}.
\]
\end{proof}

\subsection*{Outline of the construction of the finite test~$\miller{Q}$} 
Let~${Q=\{q_0 <  \cdots < q_n\}}$ be a nonempty finite subset of the domain of~$g$, where the notation used to describe~$Q$ has its obvious meaning, i.e., $Q$ is the set of~$q_0, \dots, q_n$, and~$q_i < q_{i+1}$ for all~$i$. Note that~--- in contrast to Section~\ref{proof:bounded}~--- $q_0,\dots,q_n$ don't need to be the \textit{first} $n+1$ elements of the effective enumeration of the domain of $g$. We describe the construction of the finite test~$\miller{Q}$, which is an extended version of a construction used by Miller~\cite[Lemma 1.2]{Miller-2017} in connection with left-c.e.\ reals. Using the notation defined in the previous paragraphs, for all~$i$ in~$\{0, \dots, n\}$, let
\begin{align*}
\fund{i} &= \funsymbd(q_i) = g(q_i) - d q_i,
\\
\func{i} &= \funsymbc(q_i) = g(q_i) - cq_i,
\\
\intindextwo{i}{j}
&= \intrational{q_i}{q_j}
= [\funsymbc(q_i) - \funsymbd(q_i), \funsymbc(q_j) - \funsymbd(q_i)]
= [\ee{q_i}, \func{j} - \fund{i}].
\end{align*}

The properties of the intervals of the form~${\intrational{p}{q}}$ extend to the intervals~${\intindextwo{i}{j}}$: for example, any two nonempty intervals of the form~${\intindextwo{i}{j}}$ and~${\intindextwo{i}{j^{\prime}}}$ have the same left endpoint, i.e., ${\min \intindextwo{i}{j}}$ and~${\min \intindextwo{i}{j'}}$ are the same for all~$i$, $j$, and~$j'$.

\medskip

The test~$\miller{Q}$ is constructed in successive steps~$j=0, 1, \dots, n$, where, at each step~$j$, intervals~$\intu{0}{j}, \dots, \intu{n}{j}$ are defined. Every such interval~$\intu{i}{j}$ has the form
\[
\intu{i}{j} = \intindextwo{i}{\rightendofu{i}{j}} = \intindextwo{i}{k} = \intrational{q_i}{q_k} =  [\funsymbc(q_i) - \funsymbd(q_i), \funsymbc(q_k) - \funsymbd(q_i)]
\] 
for some index~$k\in \{0, \dots, n\}$, where~$\rightendofu{\cdot}{j}$ is an index-valued function that maps every index $i$ to such index $k$ that $\intindextwo{i}{k} = \intu{i}{j}$.  

At step~$0$, for~$i=0, \dots, n$, we set the values of the function $\rightendofu{i}{0}$ by
\begin{equation}
    \rightendofu{i}{0} = i
\end{equation}
and initialize the intervals~$\intu{i}{0}$ as zero-length intervals
\begin{equation}\label{eq:step-zero}
    \intu{i}{0} = \intindextwo{i}{\rightendofu{i}{0}} = \intinit{i} = \intrational{q_i}{q_i} = [\ee{q_i}, \ee{q_i}].
\end{equation}

In the subsequent steps, every change of an interval amounts to an expansion at the right end in the sense that, for all indices~$i$, the intervals~$\intu{i}{0}, \dots, \intu{i}{n}$ share the same left endpoint, while their right endpoints are nondecreasing. More precisely, as we will see later, for~$i=0, \dots, n$, we have
\begin{align*}
\ee{q_i} &= \min \intu{i}{0} =  \cdots = \min \intu{i}{n},\\
\ee{p} &= \max \intu{i}{0} \le  \cdots \le \max \intu{i}{n}, 
\end{align*}
and thus~$\intu{i}{0} \subseteq \cdots  \subseteq \intu{i}{n}$. After concluding step~$n$, we define the finite test
\[\miller{Q} =  (\intu{0}{n}, \dots, \intu{n}{n}).\]

In case the right endpoints of two intervals of the form~$\intu{i}{j-1}$ and~$\intu{i}{j}$ coincide, we say that the interval with index~$i$ remains unchanged at step~$j$. Similarly, we will speak informally of the interval with index~$i$, or~$\intu{i}{}$, for short, in order to refer to the sequence~$\intu{i}{0}, \dots, \intu{i}{n}$ in the sense of one interval that is successively expanded. 

Due to technical reasons, for an empty set $\emptyset$, we define $\miller{\emptyset} = \emptyset$.

\subsection*{A single step of the construction and the index stair}
During step~${j>0}$, we proceed as follows. Let $t_0$ be the largest index among $\{0,\dots,j-1\}$ such that $\func{t_0} > \func{j}$, i.e., let
\begin{equation}\label{eq:define-u}
    t_0 = \argmax\{q_z: z < j \text{ and } \func{z} > \func{j}\}
\end{equation}
in case such index exists and $t_0=-1$ otherwise.

Next, define indices~$s_1, t_1, s_2, t_2, \dots$ inductively as follows. For~$h= 1, 2, \dots$, assuming that~$t_{h-1}$ is already defined, where~$t_{h-1}<j-1$, let
\begin{eqnarray}
    \label{eq:define-s-h}
    s_h &=& \max \; \argmin\{\fund{x}: t_{h-1} < x \leq j-1\},
    \\
    \label{eq:define-t-h}
    t_h &=& \max \; \argmax\{\func{y}: s_h \leq y \leq j-1\}.
\end{eqnarray}
That is, the operator~$\argmin$ yields a set of indices~$x$ such that~$\fund{x}$ is minimum among all considered values, and~$s_h$ is chosen as the largest index in this set, and similarly for~$\argmax$ and the choice of~$t_h$.

Since we assume that~$t_{h-1}<j-1$, the minimization in~\eqref{eq:define-s-h} is over a nonempty set of indices, hence~$s_h$ is defined and satisfies~$s_h\leq j-1$ by definition. Therefore, the maximization in~\eqref{eq:define-t-h} is over a nonempty index set, hence also~$t_h$ is defined.

The inductive definition terminates as soon as we encounter an index~$l\ge 0$ such that~$t_l = j-1$, which will eventually be the case by the previous discussion and because, obviously, the values~$t_0, t_1, \dots$ are strictly increasing. For this index~$l$, we refer to the finite sequence~$(t_0 ,s_1, t_1, \dots, s_l, t_l)$ (or, for short, $(t_0,s_1,t_1,\dots)$ in case  the value of $l$ is not important) as the \defhigh{index stair of step} $j$. E.g., in case~$l=1$, the index stair is~$(t_0, s_1, t_1)$, and in case~$l=0$, the index stair is~$(t_0)$. Note that~$l=0$ holds if and only if even~$s_1$ could not be defined, where the latter in turn holds if and only if~$t_0$ is equal to~$j-1$.

Next, for~$i=1, \dots, n$, we set the values of $\rightendofu{i}{j}$ and define the intervals~$\intu{i}{j}$. For a start, in case $l\ge 1$, let
\begin{eqnarray}
    \label{eq: step-definition-non-terminal:right-end}
    \rightendofu{s_1}{j} &=& {j},
    \\
    \label{eq: step-definition-non-terminal}
    \intu{s_1}{j} &=& \intindextwo{s_1}{\rightendofu{s_1}{j}} = \intindextwo{s_1}{j} =  [\func{s_1}- \fund{s_1}, \func{j} - \fund{s_1}],
\end{eqnarray}
and call this a \defhigh{nonterminal expansion} of the interval~$\intu{s_1}{}$ \defhigh{at step} $j$.
In case $l \ge 2$, in addition, let for~$h=2, \dots, l$  
\begin{eqnarray}
    \label{eq: step-definition-terminal:right-end}
    \rightendofu{s_h}{j} &=& {t_{h-1}},
    \\
    \label{eq:step-definition-terminal}
    \intu{s_h}{j} &=& \intindextwo{s_h}{\rightendofu{s_h}{j}} = \intindextwo{s_h}{t_{h-1}} =  [\func{s_h} - \fund{s_h}, \func{t_{h-1}} - \fund{s_h}],
\end{eqnarray}
and call this a \defhigh{terminal expansion} of the interval~$\intu{s_h}{}$ \defhigh{at step} $j$.

For all remaining indices, the interval with index~$i$ \defhigh{remains unchanged at step}~$j$, i.e., for all $i\in\{0, \dots, n\} \setminus \{s_1, \dots, s_l\}$, let
\begin{eqnarray}
    \label{eq:unchanged:right-end}
    \rightendofu{i}{j} &=& \rightendofu{i}{j-1},
    \\
    \label{eq:unchanged}
    \intu{i}{j} &=& \intu{i}{j-1}.
\end{eqnarray}

The choice of the term \say{terminal expansion} is motivated by the fact that, in case a terminal expansion occurs for the interval with index~$i$ at step $j$, then, at all further steps $j+1,\dots,n$, the interval remains unchanged, as we will see later.

We conclude step~$j$ by defining for~$i=0,\dots,n$ the half-open interval
\begin{equation}
\label{eq:definition-of-V}
\intv{i}{j} = \intu{i}{j} \setminus  \intu{i}{j-1}.
\end{equation}
That is, during step~$j$, the interval with index~$i$ is expanded by adding at its right end the half-open interval~$\intv{i}{j}$, i.e., we have
\begin{equation}\label{eq:basic-properties-of-V}
U^{j}_i = U^{j-1}_i \cupdisjoint V^j_i \makebox[6em]{ where }
|U^{j}_i| = |U^{j-1}_i| + |V^j_i| . 
\end{equation}
This includes the degenerated case where the interval with index~$i$ is not changed, hence~$V^j_i$ is empty and has length~$0$.

In what follows, in connection with the construction of a test of the form~$\miller{Q}$, when appropriate, we will occasionally write~$t_0^j$ for the value of~$t_0$ chosen during step~$j$ and similarly for other values like~$s_h$ in order to distinguish the values chosen during different steps of the construction. 

\subsection*{The proof of Claim~\ref{infinite-covering}}
Now, as the construction of the tests of the form~$\miller{Q}$ has been specified,  we can already demonstrate Claim~\ref{infinite-covering}. Let $Q_0 \subseteq Q_1 \subseteq \dots$ be a sequence of sets that converges to the domain of~$g$ as in the assumption of the claim. Any finite subset~$H$ of the domain of~$g$  will be a subset of~$Q_n$ for all sufficiently large indices~$n$, where then, for all such~$n$, it holds that~$\coverset{H}{\ee{\beta}} \le \coverallfun{Q_n}(\ee{\beta})$ by definition of~$\coverallfun{Q_n}$. Consequently, in order to show Claim~\ref{infinite-covering}, i.e., that the values $\coverall{Q_n}{\ee{\beta}}$ tend to infinity, it  suffices to show that the function~$H \mapsto \coversetfun{H}(\ee{\beta})$ is unbounded on the finite subsets~$H$ of the domain of~$g$.  

Recall that we have defined subsets~$S$ and~$T$ of the domain of~$g$, which contain only rationals~$q < \beta$. Let~$r_0 < r_1 < \dots$ be a sequence such that, for all indices~$i \ge 0$, it holds that
\begin{equation}\label{eq:proof-infinite-covering}
 r_{2i} \in T, \qquad\qquad  
 r_{2i+1} \in S, \qquad\qquad 
\funsymbc(r_{2i+1}) < \funsymbc(r_{2i+2}) <  \funsymbc(r_{2i}).   
\end{equation}

Such a sequence can be obtained by the following nonconstructive inductive definition.  Let~$r_0$ be an arbitrary number in~$T$. Assuming that~$r_{2i}$ has already been defined, let~$r_{2i+1}$ be equal to some~$r$ in~$S$ that is strictly larger than~$r_{2i}$. Note that such~$r$ exists since~$r_{2i} < \beta$, and~$\beta$ is an accumulation point of~$S$. Furthermore, assuming that~$r_{2i}$ and~$r_{2i+1}$ have already been defined, let~$r_{2i+2}$ be equal to some~$r$ in~$T$ that is strictly larger than~$r_{2i+1}$ and such that the second inequality in~\eqref{eq:proof-infinite-covering} holds. Note that such~$r$ exists because, by definition of~$T$, we have~$  \funsymbc(r_{2i}) > \alpha - c \beta$, while~$\beta$ is also an accumulation point of~$T$, and~$\funsymbc(r)$ converges to~$\alpha - c \beta$ when~$r$ tends nondecreasingly to~$\beta$. Finally, observe that the first inequality in~\eqref{eq:proof-infinite-covering} holds automatically for~$r_{2i+1}$ in~$S$ and~$r_{2i+2}$ in~$T$ because, by Claim~\ref{claim:set-s-and-t-are-disjoint}, the set~$S$ is disjoint from~$T$, hence, by definition of~$T$, we have   
\[\funsymbc(r_{2i+1}) \le \alpha - c \beta < \funsymbc(r_{2i+2}).
\]
Now, let~$H$ be equal to~$\{r_0, r_1, \dots, r_{2k}\}$, and consider the construction of~$M(H)$. For the remainder of this proof, we will use the indices of the~$r_j$ in the same way as the indices of the~$q_j$ are used in the description of the construction above. For example, for~$i=0, \dots, k-1$, during step~$2i+2$ of the construction of~$M(H)$, the index~$t_0$ is chosen as the maximum index~$z$ in the range~$0, \dots, 2i+1$ such that~$\funsymbc(r_{2i+2}) < \funsymbc(r_{z})$. By~\eqref{eq:proof-infinite-covering}, this means that, in step~$2i+2$, the index~$t_0$ is set equal to~${2i}$ and --- since~$2i+1$ is the unique index strictly between~$2i$ and~$2i+2$ --- the index stair of this step is~$({2i}, {2i+1}, {2i+1})$.   Accordingly, by construction, the interval~$\intu{2i+1}{2i+2}$ coincides with the interval~$\intrational{r_{2i+1}}{r_{2i+2}}$. By Claim~\ref{claim:intervals-with-endpoints-in-s-and-t}, this interval, and thus also its superset~$\intu{2i+1}{2k}$, contains~$\ee{\beta}$. The latter holds for all~$k$ different values of~$i$, hence~$\coverset{H}{\ee{\beta}} \ge k$. This concludes the proof of Claim~\ref{infinite-covering} since~$k$ can be chosen arbitrarily large.

\subsection*{Some properties of the intervals~$\intu{j}{i}$} We gather some basic properties of the points and intervals that are used in the construction.

\begin{claimbbb}\label{claim:p-q-sequence}
Let~$Q= \{q_0 <  \cdots < q_n\}$ be a subset of the domain of~$g$. Consider some step~$j$ of the construction of~$\miller{Q}$, and let~$(t_0, s_1,t_1,\dots,s_l,t_l)$ be the corresponding index stair. Then we have~$\func{j} < \func{t_0}$ in case~$t_0\neq -1$.

In case the index~$s_1$ could not be defined, i.e., in case~$l=0$, we have~$t_0=j-1$. Otherwise, i.e., in case~$l>0$, we have
\begin{align}
    \label{eq:stair-positions}
    &t_0 < s_1 < t_1 < \cdots < s_l \leq  t_l = j-1 < j,
    \\
    \label{eq:stair-values}
    &\fund{s_1} < \cdots < \fund{s_l} < \func{t_l} < \cdots < \func{t_1}\le \func{j}.
\end{align}

\end{claimbbb}

\begin{proof}
The assertion on the relative size of~$\func{j}$ and~$\func{t_0}$ is immediate by definition of~$t_0$. In case~$s_1$ cannot be defined, the range between~$t_0$ and~$j$ must be empty, and~$t_0=j-1$ follows. Next, we assume~$l>0$ and demonstrate~\eqref{eq:stair-positions} and~\eqref{eq:stair-values}. By definition of the values~$s_h$ and~$t_h$, it is immediate that we have~$s_h \le t_h < s_{h+1}$ for all~$h\in\{1,\dots,l-1\}$ and have~$s_l \le t_l = j-1$. In order to complete the proof of~\eqref{eq:stair-positions}, assume~$s_h = t_h$ for some~$h$. Then we have 
\begin{equation}\label{eq:earlier-coincidence}
    \ee{q_{s_h}} = \func{s_h} - \fund{s_h}= \func{t_h} - \fund{s_h} \ge \func{j-1} - \fund{j-1} = \ee{q_{j-1}},
\end{equation}
where the inequality holds true because~$\func{t_h} \ge \func{j-1}$ and~$\fund{s_h} \le \fund{j-1}$ hold for all~$h$. So we obtain~$s_h = t_h = j-1$, and thus~$h = l$ because, otherwise, i.e., in case~$s_h < j-1$, we would have~$q_{s_h}< q_{j-1}$.

By definition of~$s_1$ and~$l$, it is immediate that, in case~$l=0$, we have~$t_0=j-1$.

It remains to show~\eqref{eq:stair-values} in case~$l >0$. The inequality~$\func{t_1} \le \func{j}$ holds because its negation would contradict the choice of $t_0$ in the range~$0, \dots, j-1$ as largest index with maximum $\funsymbc$-value, as we have~$t_0 < t_1 < j$ by~\eqref{eq:stair-positions}. In order to show~$\fund{s_l} < \func{t_l}$, it suffices to observe that we have~$\fund{s_l} \le \fund{t_l}$ by choice of~$s_l$ and~$s_l \le t_l < j$ and know that~$\fund{t_l} < \func{t_l}$ from Claim~\ref{claim:properties-gamma-delta}. In order to show the remaining strict inequalities, fix~$h$ in~$\{1, \dots, l-1\}$. By choice of~$s_h$, we have~$\fund{s_h} < \fund{x}$ for all~$x$ that fulfill~$s_h < x \leq j-1$, and since~$s_{h+1}$ is among these~$x$, it holds~$\delta_{s_h} < \delta_{s_{h+1}}$. By a similar argument, it follows that~$\func{t_{h+1}} < \func{t_{h}}$. 
\end{proof}

\begin{claimbbb}\label{claim:before}
Let~$Q= \{q_0 <  \cdots < q_n\}$ be a subset of the domain of~$g$, and consider the construction of~$\miller{Q}$. Let~$i$ be in~$\{0, \dots, n\}$. Then it holds that
\begin{equation}\label{eq:before-before-step-i}
\intu{i}{0} = \cdots = \intu{i}{i}.
\end{equation}
Furthermore, for all steps~$j \ge  i$ of the construction, it holds that
\begin{equation}\label{eq:before-beyond-step-i}
\intu{i}{j} = \intindextwo{i}{x} \makebox[7em]{ where } x \le j.
\end{equation}
\end{claimbbb}

\begin{proof}
The equalities in~\eqref{eq:before-before-step-i} hold because the index stair of every step~${j  \le i}$ contains only indices that are strictly smaller than~$j$, and thus also than~$i$, hence, by~\eqref{eq:unchanged}, the interval with index~$i$ remains unchanged at all such steps. 

Next, we demonstrate~\eqref{eq:before-beyond-step-i} by induction over all steps~$j \ge i$. The base case ${j=i}$ follows from~\eqref{eq:before-before-step-i} and because, by definition, we have~$\intu{i}{0} = \intindextwo{i}{i}$. At the step~$j>i$, we consider its index stair~$(t_0, s_1, t_1,\dots, s_l, t_l)$. Observe that all indices that occur in the index stair are strictly smaller than~$j$. The induction step now is immediate by distinguishing the following three cases. In case~$i=s_1$, we have~$\intu{i}{j}= \intindextwo{i}{j}$. In case~$i=s_h$ for some~$h >1$, we have~$\intu{i}{j}= \intindextwo{i}{t_{h-1}}$. In case~$i$ differs from all indices of the form~$s_h$, by~\eqref{eq:unchanged}, the interval with index~$i$ remains unchanged at step~$j$, and we are done by the induction hypothesis.  
\end{proof}

\begin{claimbbb} \label{claim:previous-step}
Let~$Q= \{q_0 <  \cdots < q_n\}$ be a subset of the domain of~$g$, and consider the construction of~$\miller{Q}$. Let~$j\geq 1$ be a step of the construction where at least the index~$s_1$ could be defined, and let~$(t_0,s_1,t_1, \dots, s_l,t_l)$ be the index stair of this step. Then, for~$h=1, \dots, l$, we have
\begin{equation}\label{eq:formulas-for-previous-step}
\intu{s_{h}}{j-1} = \intindextwo{s_{h}}{t_{h}}, 
\makebox[10em]{ hence, in particular, }
\max \intu{s_{h}}{j-1} = \func{t_h} - \fund{s_h}.
\end{equation}
Consequently, for~$i=0, \dots, n$, we have
\begin{equation}
    \intu{i}{0}  \subseteq \cdots  \subseteq \intu{i}{n}, \makebox[6em]{wherein} \max U_i^0\leq \max U_i^1 \leq \dots \leq \max U_i^n.
\end{equation}
\end{claimbbb}

\begin{proof}
In order to prove the claim, fix some~$h$ in~$\{1,\dots,l\}$. In case $s_h = t_h$, by~\eqref{eq:stair-positions}, we have~$h=l$ and~$s_h = t_h = j-1$, hence~\eqref{eq:formulas-for-previous-step} holds true because, by construction and~\eqref{eq:before-before-step-i}, we have
\[
\intu{s_h}{j-1} = \intu{s_h}{s_h} = \intu{s_h}{0}=\intindextwo{s_h}{s_h}=\intindextwo{s_h}{t_h}.
\]
So we can assume the opposite, i.e., that~$s_h$ and~$t_h$ differ. We then obtain
\begin{equation}\label{eq:utj-lessthan-sj}
t_{h-1} \le t_0^{t_h} < s_h < t_h < j,
\end{equation}
where~$t_0^{t_h}$, as usual, denotes the first entry in the index stair of step~$t_h$. Here, the last two strict inequalities are immediate by~Claim~\ref{claim:p-q-sequence} since $s_h$ differs from~$t_h$. In case the first strict inequality was false, again, by Claim~\ref{claim:p-q-sequence}, we would have $s_h \le t_0^{t_h} < t_h < j$ as well as~$\func{t_h} < \func{t_0^{t_h}}$, which together contradict the choice of~$t_h$. Finally, the first inequality obviously holds in case~$h=1$ and~$t_{0}=-1$. Otherwise, 
we have~$\func{t_{h-1}} > \func{t_{h}}$ by~\eqref{eq:stair-values} as well as~$t_{h-1} < t_{h}$, hence, by definition, the value~$t_0^{t_h}$ will not be chosen strictly smaller than~$t_{h-1}$.  

By~\eqref{eq:utj-lessthan-sj}, it follows that
\[
\{x \colon t_0^{t_h} < x < t_h\} \subseteq 
\{x \colon t_{h-1} < x < j\}.
\]
By definition, the index~$s_1^{t_h}$ is chosen as the largest~$x$ in the former set that minimizes~$\fund{x}$, while~$s_h$ is chosen from the latter set by the same condition, i.e., as the largest~$x$ that minimizes~$\fund{x}$. Again, by~\eqref{eq:utj-lessthan-sj}, the index~$s_h$ is also in the former set, therefore, it must be the largest index minimizing~$\fund{x}$ there. So we have~$s_1^{t_h} = s_h$, hence~$\intu{s_h}{t_h}= \intindextwo{s_h}{t_h}$ follows from construction. 

Next, we argue that $U^{t_h}_{s_h} = U^{t_h}_{s_h}$ by demonstrating that
\[
\intu{s_h}{t_h} = \intu{s_h}{t_h+1} = \dots = \intu{s_h}{j-1},\]
i.e., that at all steps~$y=t_h+1, \dots, j-1$, the interval~$\intu{s_h}{}$ remains  unchanged. For every such step~$y$, by definition of~$t_{h}$, we have~$\func{y} < \func{t_{h}}$, hence~$s_h < t_h \le t_0^y$ by choice of~$t_0^y$. Consequently, the index~$s_h$ does not occur in the index stair of step~$y$, and we are done by~\eqref{eq:unchanged}.

We conclude the proof of the claim by showing for~$i=1, \dots, n$ the inequality
\[\max \intu{i}{j-1} \le \max \intu{i}{j},\]
which then implies~$\intu{i}{0}\subseteq \cdots \subseteq \intu{i}{n}$ because, by construction, the latter intervals all share the same left endpoint~$\min \intu{i}{0} = \ee{q_i}$, and~$j$ is an arbitrary index in~$\{1, \dots, n\}$.

For indices~$i$ that are not equal to some~$s_h$, the interval~$i$ remains unchanged at step~$j$, and we are done. So we can assume~$i = s_h$ for some~$h$ in~$\{1, \dots, l\}$; thus,~$\max \intu{i}{j-1} = \func{t_h}- \fund{s_h}$ follows from~\eqref{eq:formulas-for-previous-step}. The value~$\func{t_h}$ is strictly smaller than both values~$\func{j}$ and~$\func{t_{h-1}}$ by choice of~$t_0$ and~$t_{h-1}$. So we are done because, by construction, in case~$h=1$, we have~$\max \intu{i}{j}= \func{j}- \fund{s_h}$, while, in case~$h >1$, we have~$\max \intu{i}{j}= \func{t_{h-1}}- \fund{s_h}$.
\end{proof}

As a corollary of Claim~\ref{claim:previous-step}, we obtain that, when constructing a test of the form~$\miller{Q}$, any terminal expansion of an interval at some step is, in fact, terminal in the sense that the interval will remain unchanged at all larger steps. 

\begin{claimbbb}
Let~$Q= \{q_0 <  \cdots < q_n\}$ be a subset of the domain of~$g$, and consider the construction of~$\miller{Q}$. Let~$j\geq 1$ be a step of the construction, where the index~$s_2$ could be defined, and let~$(t_0,s_1,t_1, \dots, s_l,t_l)$ be the index stair of this step. Then, for every $h=2, \dots,l$, it holds that $\rightendofu{s_h}{j} = \rightendofu{s_h}{n}$, and therefore, that $\intu{s_h}{j} = \intu{s_h}{n}$.
\end{claimbbb}

\begin{proof}
For a proof by contradiction, we assume that the claim assertion is false, i.e., we can fix some~$h\ge 2$ such that the values~$\rightendofu{s_h}{j}$ and~$\rightendofu{s_h}{n}$ differ. Let~$k$ be the least index in~$\{j+1,\dots, n\}$ such that the values~$\rightendofu{s_h}{k-1}$ and~$\rightendofu{s_h}{k}$ differ, and let~$(t_0^k, s^k_1, t^k_1, \dots)$ be the index stair of step~$k$. Since the interval with index~$s_h$ does not remain unchanged at step~$k$, we must have~$s_h = s^k_x$ for some $x\geq 1$. In order to obtain the desired contradiction, we distinguish the cases~$x=1$ and~$x>1$. In case~$x=1$, by construction, we have
\[t_0^k < s^k_1 = s_h < j < k
\makebox[7em]{ and } t_0^k \le t_0 < s_1 < j < k,
\]
where all relations are immediate by choice of the involved indices except the nonstrict inequality. The latter inequality holds by choice of~$t_0^k$ because, by the chain of relations on the left, we have~$t_0^k < j$, and thus~$\func{j} \le \func{k}$, while~$\func{i} \le \func{j}$ holds for~$i=t_0 +1, \dots, j-1$ by choice of~$t_0$. Now, we obtain as a contradiction that~$s^k_1=s_h$ is chosen in the range~$t^k_0 +1, \dots, k-1$ as largest index that has minimum $\funsymbd$-value, where this range includes~$s_1$, hence~$\fund{s^k_1} \le \fund{s_1}$, while~$\fund{s_1} < \fund{s_h}$ by~$h\ge 2$. 

In case~$x>1$, we obtain
\begin{equation}\label{eq:int-i-int-u-int-i}
t_{h-1} = \rightendofu{s_h}{j} = \rightendofu{s_h}{k-1} = t^k_{x}, 
\end{equation}
which contradicts to~$t_{h-1}< s_h = s^k_x  \le t^k_x$. The equalities in~\eqref{eq:int-i-int-u-int-i} follow, from left to right, from~$h\ge 2$, from the minimality condition in the choice of~$k$ and, finally, from~$s_h=s^k_x$ and Claim~\ref{claim:previous-step}.
\end{proof}

The explicit description of the intervals of the form~$\intu{s_h}{j-1}$ according to Claim~\ref{claim:previous-step} now yields an explicit description of the endpoints of the half-open intervals of the form~$\intv{i}{j}$, from which in turn we obtain that all such intervals occurring at the same step are mutually disjoint, and the sum of their measures is equal to $\func{j} - \func{j-1}$.

\begin{claimbbb}\label{V=U-U}
Let~$Q= \{q_0 <  \cdots < q_n\}$ be a subset of the domain of~$g$, and consider the construction of~$\miller{Q}$. Let~$j>0$ be a step of the construction.

If $\func{j-1}\leq\func{j}$, then it holds for the index stair $(t_0,s_1,t_1, \dots, s_l,t_l)$ of this step that $l>0$, i.e., that $s_1$ can be defined, and we have
\begin{eqnarray}
\label{eq:first-added-interval}
\intv{s_1}{j} &=& (\func{t_1} - \fund{s_1}, \func{j} - \fund{s_1}],\\
\label{eq:m-th-added-interval}
\intv{s_h}{j}&=& (\func{t_h} - \fund{s_h}, \func{t_{h-1}} - \fund{s_h}]
\quad \text{ for }h\ge 2 \quad(\text{if defined}),\\
\label{eq:empty-added-interval}
\intv{i}{j}&=& \emptyset
\quad\text{ for~$i$ in }  \{0,\dots,n\}\setminus\{s_1,\dots,s_l\}.
\end{eqnarray}
In particular, the half-open intervals~$\intv{0}{j}, \dots, \intv{n}{j}$ are mutually disjoint, and the sum of their Lebesgue measures can be bounded as follows
\begin{equation}\label{eq:summarized-length-for-step-j}
    \sum_{i=0}^n \mu(\intv{i}{j}) = \sum_{h=1}^l \mu(\intv{s_h}{j}) = \func{j} - \func{j-1}.
\end{equation}

If $\func{j-1}>\func{j}$, then the index stair of this step has a form $(j-1)$, i.e., $t_0 = j-1$, $l=0$, all the intervals~$\intv{0}{j}, \dots, \intv{n}{j}$ are empty, i.e.,
\begin{equation}
\label{eq:empty-added-interval-empty-step}
\intv{i}{j}= \emptyset\text{ for all }i,
\end{equation}
and the sum of their Lebesgue measures is equal to zero
\begin{equation}\label{eq:summarized-length-for-empty-step}
    \sum_{i=0}^n \mu(\intv{i}{j}) = 0.
\end{equation}
\end{claimbbb}

\begin{proof}
If~$\func{j-1}\leq\func{j}$, then we have $t_0 \neq j-1$, hence the set $\{x: t_0 < x\leq j-1\}$ used in~\eqref{eq:define-s-h} to define $s_1$ contains at least one index, namely $j-1$, and therefore, $s_1$ can be defined.

If~$\func{j-1}>\func{j}$, then we have $t_0 = j-1$, hence the set $\{x: t_0 < x\leq j-1\}$ is empty, and $s_0$ cannot be defined.

Recall that, by construction, the intervals~$\intu{i}{0}, \dots, \intu{i}{n}$ are all nonempty and have all the same left endpoint~$\func{i}-\fund{i}$; thus, we have
\[
\intv{i}{j} 
= \intu{i}{j} \setminus \intu{i}{j-1}
= (\max \intu{i}{j-1}, \max \intu{i}{j}].
\]
This implies~\eqref{eq:empty-added-interval} in case~$\func{j-1}\leq\func{j}$ and~\eqref{eq:empty-added-interval-empty-step} in case~$\func{j-1}>\func{j}$ since, for~$i$ not in~$\{s_1, \dots, s_l\}$, the interval with index~$i$ remains unchanged at step~$j$, hence~$\intv{i}{j}$ is empty.

In case~$\func{j-1}>\func{j}$, we obtain~\eqref{eq:summarized-length-for-empty-step} directly from~\eqref{eq:empty-added-interval-empty-step} by
\[\sum_{i=0}^n \mu(\intv{i}{j}) = \sum_{i=0}^n \mu(\emptyset) = 0,\]
so, from now on, we assume that~$\func{j-1}\leq\func{j}$ and, as we have seen before, $l>0$.

In order to obtain~\eqref{eq:first-added-interval} and~\eqref{eq:m-th-added-interval} in this case, it suffices  to observe that~$\max \intu{s_h}{j}$ is equal to~$\func{j} - \fund{s_1}$ in case~$h=1$ and is equal to~$\func{t_{h-1}} - \fund{s_h}$ in case~$h\ge 2$, respectively, while~$\max \intu{s_h}{j-1} = \func{t_h}-\fund{s_h}$ for~$h=1, \dots, l$ by Claim~\ref{claim:previous-step}. 

Next, we show that the half-open intervals~$\intv{0}{j}, \dots, \intv{n}{j}$ are mutually disjoint. These intervals are all empty except for~$\intv{s_1}{j}, \dots, \intv{s_l}{j}$. In case the latter list contains at most one interval, we are done. So we can assume~$l\ge 2$. Disjointedness of~$\intv{0}{j}, \dots, \intv{n}{j}$ then follows from
\[
\min \intv{s_l}{j} <  \max \intv{s_l}{j} 
< \cdots < \min \intv{s_1}{j} < \max \intv{s_1}{j}.
\]
These inequalities hold  because, for~$h=  2, \dots, l$, by Claim~\ref{claim:p-q-sequence}, we have~${\func{t_{h-1}}>\func{t_{h}}}$ and~$\fund{s_{h-1}}<\fund{s_{h}}$, which together with~\eqref{eq:first-added-interval} and~\eqref{eq:m-th-added-interval} yields
\[
\func{t_{h}} - \fund{s_{h}} = \min \intv{s_h}{j} < \max \intv{s_h}{j} = \func{t_{h-1}} - \fund{s_h} <  
\func{t_{h-1}} - \fund{s_{h-1}} =  \min \intv{s_{h-1}}{j}.
\]
Since the intervals~$\intv{0}{j}, \dots, \intv{n}{j}$ are mutually disjoint, the Lebesgue measure of their union is equal to
\begin{align*}
\sum_{i=0}^n \mu(\intv{i}{j})
=\sum_{h=1}^l \mu(\intv{s_h}{j})
&=\mu(\intv{s_1}{j})  \;\; + \sum_{h=2}^l \mu(\intv{s_h}{j})\\  
&= (\func{j} - \func{t_1}) + \sum_{h=2}^l (\func{t_{h-1}} - \func{t_{h}})
= \func{j} - \func{t_l} = \func{j} - \func{j-1},
\end{align*}
where the last two equalities are implied by evaluating the telescoping sum and because~$t_l$ is equal to~$j-1$ by Claim~\ref{claim:p-q-sequence}, respectively.
\end{proof}

\subsection*{The proof of Claim~\ref{claim:upper-bound-for-single-stage}}
Using the results on the intervals~$\intv{i}{j}$ in Claim~\ref{V=U-U}, we can now easily demonstrate Claim~\ref{claim:upper-bound-for-single-stage}. We have to show for every subset $Q=\{q_0 <\cdots < q_n\}$ of the domain of~$g$ that
\begin{equation}\label{eq:finite-measure-inductive-form}
\mu\big(\miller{Q}\big)\leq g(q_n) - g(q_0).
\end{equation}
This inequality holds true because we have
\begin{align*}
\mu\big(\miller{Q}\big) 
&= \sum_{U \in \miller{Q}} \mu(U) 
= \sum_{i=0}^{n} \mu(\intu{i}{n}) 
= \sum_{i=0}^{n} \sum_{j=1}^{n} \mu(\intv{i}{j}) 
= \sum_{j=1}^{n} \sum_{i=0}^{n}  \mu(\intv{i}{j}) \\
&= \sum_{j=1}^{n}  \big( \max\{\func{j} - \func{j-1},0\}\big) \leq g(q_n) - g(q_0).
\end{align*}
In the first line, the first equality holds by definition of~$\mu\big(\miller{Q}\big)$, while the second and the third equalities hold by construction of~$\miller{Q}$ and by~\eqref{eq:definition-of-V}, respectively.

In the second line, the equality holds because, for every $j$, we have
\[\sum_{i=0}^{n}  \mu(\intv{i}{j}) = \max\{\func{j} - \func{j-1},0\}\]
due to the following argumentation: in case $\func{j-1}\leq \func{j}$, we obtain from Claim~\ref{V=U-U},~\eqref{eq:summarized-length-for-step-j}, that $\sum_{i=0}^n \mu(\intv{i}{j}) = \func{j} - \func{j-1} \geq 0$, and in case $\func{j-1}> \func{j}$, we obtain from Claim~\ref{V=U-U},~\eqref{eq:summarized-length-for-empty-step}, that $\sum_{i=0}^n \mu(\intv{i}{j}) = 0$. 

Finally, the inequality in the second line holds because the difference ${g(q_n)-g(q_0)}$ can be rewritten as a telescoping sum
\[g(q_n) - g(q_0) = \big(g(q_n) - g(q_{n-1})\big) + \big(g(q_{n-1}) - g(q_{n-2})\big) + \dots + \big(g(q_1) - g(q_0)\big),\]
and for every $j$ from $1$ to $n$, we have
\[\max\{\func{j} - \func{j-1},0\}\leq g(q_j) - g(q_{j-1})\]
due to the following argumentation: in case $\func{j-1} \leq \func{j}$, we have
\[
0 \leq \func{j} - \func{j-1}
= \big(g(q_j) - c q_{j}\big) -  \big(g(q_{j-1}) - c q_{j-1}\big) \le g(q_j) - g(q_{j-1}),
\]
 where the equality holds since~$\func{k} = \funsymbc(q_k) = g(q_j) - c q_{k}$ for every $k$ in the range ${0,\dots,n}$ and the right inequality is implied by~$q_{j-1}<q_{j}$. In case $\func{j-1}>\func{j}$, we directly have
\[\func{j} - \func{j-1} < 0\leq g(q_j) - g(q_{j-1}),\]
where the right inequality is implied by monotonicity of $g$ for arguments ${q_{j-1}<q_j}$.

\subsection*{Preliminaries for the proof of Claim~\ref{general-bound-for-coverall}}
The following claim asserts that, when adding to a finite subset~$Q$ of the domain of~$g$ one more rational that is strictly larger than all members of~$Q$,  the cover function of the test corresponding to~$Q$ increases at most  by one on all nonrational arguments.
\begin{claimbbb} \label{claim:k<k<k+1}
Let~$Q$ be a finite subset of the domain of~$g$. Then, for every real $p\in[0,1]$, it holds that
\begin{equation}\label{eq:k<k<k+1}
\coverset{Q \setminus \{\max Q\}}{p}\leq \coverset{Q}{p}\leq \coverset{Q \setminus \{\max Q\}}{p} + 1.
\end{equation}
\end{claimbbb}
\begin{proof}
Let~$Q = \{q_0 <  \cdots < q_n\}$ be a finite subset of the domain of~$g$. We consider the constructions of the tests~$\miller{Q\setminus \{q_n\}}$ and~$\miller{Q}$ and denote the intervals constructed in the latter test by~$\intu{i}{j}$, as usual. The steps~$0$ through~$n$ of both constructions are essentially identical up to the fact that, in the latter construction, in addition, the interval~$\intu{n}{0}$ is initialized as~$[\ee{q_n},\ee{q_n}]$ in step~$0$ and then remains unchanged. Accordingly, the test~$\miller{Q\setminus \{q_n\}}$ consists of the intervals~$\intu{0}{n-1}, \dots, \intu{n-1}{n-1}$, therefore, the first inequality in~\eqref{eq:k<k<k+1} holds true because the test~$\miller{Q}$ is then obtained by expanding these intervals. More precisely, in the one additional step of the construction of~$\miller{Q}$, these intervals and the interval~$\intu{n}{n-1}=\intu{n}{0}$ are expanded by letting
\[
\intu{i}{n}= \intu{i}{n-1} \cup \intv{i}{n} \quad \text{ for $i=0, \dots, n$}. 
\]
The intervals~$\intv{0}{n}, \dots, \intv{n}{n}$ are mutually disjoint by Claim~\ref{V=U-U}. Consequently, the cover functions of both tests can differ at most by one, hence also the second inequality in~\eqref{eq:k<k<k+1} holds true.
\end{proof}

The following three somewhat technical claims will be used in the proof of Claim~\ref{general-bound-for-coverall}.

\begin{claimbbb} \label{i<j<k-coverage}
Let~$Q= \{q_0 <  \cdots < q_n\}$ be a subset of the domain of~$g$, and
let $p\in (0,1]$ be a real number. Let~$i,j,k$ be indices such that~${q_0 \le  q_i<q_j<q_k < p}$, 
\begin{equation} \label{eq:i<j<k-coverage-conditions}
\ee{p}  < \func{i} - \fund{j},  \makebox[5em]{ and }\ee{p}  < \func{k} - \fund{j}. 
\end{equation}
Let~$Q_i=\{q_0,\dots,q_i\}$ and ${Q_k=\{q_0,\dots,q_i,\dots,q_j,\dots,q_k\}}$. Then the following strict inequality holds:
\begin{equation}
\coverset{Q_i}{\ee{p}}<\coverset{Q_k}{\ee{p}}.
\end{equation}
\end{claimbbb}

\begin{proof}
Let $s=\max\argmin\{\fund{x}:i<x<k\}$, and let 
\[
i'=\max\argmax\{\func{y}:i\leq y<s\} \makebox[4em]{ and }
k'=\max\argmax\{\func{y}:s<y\leq k\}.
\]
The following inequalities are immediate by definition
\begin{equation}\label{eq:s-t-greater-than-i-k}
\fund{s}\leq\fund{j},\qquad\func{s}<\func{i}\leq\func{i'},\qquad\func{s}<\func{k} \leq \func{k'},
\end{equation}
except the two strict upper bounds for~$\func{s}$. The first of these bounds, i.e., $\func{s}< \func{i}$, follows from
\[
\func{s}-\fund{s} =\ee{q_s} < \ee{p} < \func{i} - \fund{j}
\le \func{i} - \fund{s},
\]
where the inequalities hold, from left to right, by~$q_s < q_k <p$,  by~\eqref{eq:i<j<k-coverage-conditions}, and by~\eqref{eq:s-t-greater-than-i-k}. By an essentially identical argument, this chain of relations remains valid when~$\func{i}$ is replaced by~$\func{k}$, which shows the second bound, i.e.,~$\func{s}< \func{k}$.

We denote the intervals that occur in the construction of~$\miller{Q}$ by~$\intu{i}{j}$, as usual. As in the proof of Claim~\ref{claim:k<k<k+1}, we can argue that the construction of the test~$\miller{Q_i}$ is essentially identical to initial parts of the construction of~$\miller{Q_k}$ and of~$\miller{Q}$, and that a similar remark holds for the tests~$\miller{Q_k}$ and~$\miller{Q}$. Accordingly, we have
\[
\miller{Q_i} = (\intu{0}{i}, \dots, \intu{i}{i})
\makebox[5em]{ and }
\miller{Q_k} = (\intu{0}{k}, \dots, \intu{i}{k}, \intu{i+1}{k}, \dots,\intu{k}{k}),
\]
For~$x=1, \dots, i$, the interval~$\intu{x}{i}$ is a subset of~$\intu{x}{k}$ by~$i<k$ and Claim~\ref{claim:previous-step}. Hence it suffices to show
\begin{equation}\label{eq:u-s-k-assertion}
\ee{p}\in\intu{s}{k},
\end{equation}
because the latter statement implies by~$i < s < k$ that 
\[
\coverset{Q_i}{\ee{p}}+1 
\leq 
\coverset{Q_k}{\ee{p}}.
\]
We will show~\eqref{eq:u-s-k-assertion} by proving that~$\ee{p}$ is strictly larger than the left endpoint and is strictly smaller than the right endpoint of the interval~$\intu{s}{k}$. The assertion about the left endpoint, which is equal to~$\func{s}-\fund{s}=\ee{q_s}$, holds true because the inequalities~$s < k$ and~$q_k < p$ imply together that~$q_s < p$. 

In order to demonstrate the assertion about the right endpoint, we distinguish two cases.

\underline{Case 1}: $\func{i'}>\func{k'}$. In this case, let~$(t_0,s_1, t_1, \dots)$ be the index stair of the step $k'$. Then we have
\begin{equation}\label{eq:u-s-k-assertion-case-i}
i \leq  i^{\prime} \leq t_0 < s < k^{\prime} \le k,
\end{equation}
where all inequalities are immediate by choice of~$i^{\prime}$ and~$k^{\prime}$ except the second and the third one. Both inequalities follow from the definition of~$t_0$: the second one together with the case assumption, the third one because, by~$\func{s} < \func{k^{\prime}}$ and by choice of~$k^{\prime}$, no value among~$\func{s}, \dots, \func{k^{\prime}-1}$ is strictly larger than~$\func{k^{\prime}}$.

By~\eqref{eq:u-s-k-assertion-case-i}, it is immediate that the set~${\{t_0+1,\dots,k'-1\}}$ contains~$s$ and is a subset of the set~${\{i+1,\dots,k-1\}}$. By definition, the indices~$s_1$ and~$s$ minimize the value of $\fund{j}$ among the indices~$j$ in the former and in the latter set, respectively, hence we have~$s=s_1$. By construction, in step~$k^{\prime}$, the right endpoint of the interval~$\intu{s}{k'}$ is then set to~$\func{k^{\prime}} - \fund{s}$. So we are done with Case~1 because we have
\[
\ee{p} < \func{k}-\fund{j} 
\leq \func{k'} - \fund{s} 
= \max \intu{s}{k'} \le \max \intu{s}{k},
\]
where the first inequality holds by assumption of the claim, the second one holds by~\eqref{eq:s-t-greater-than-i-k}, and the last one holds by~$k' \le k$ and Claim~\ref{claim:previous-step}.

\underline{Case 2}: $\func{i'}\leq\func{k'}$. In this case, let
\begin{equation}\label{eq:index-t-used-in-proof}
        r=\min\{y:s<y\leq k\wedge \func{i'}\leq\func{y}\},
\end{equation}
and let~$(t_0,s_1, t_1, \dots)$ be the index stair of the step~$r$. By choice of~$s$ and by~$r \le k$, all values among~$\fund{s+1}, \dots, \fund{r-1}$ are strictly larger than~$\fund{s}$, hence we have~$s_1 \le s$ by choice of~$s_1$. Accordingly, the index 
\begin{equation}\label{eq:index-m-used-in-proof}
m=\max\{h > 0: s_h\leq s\}
\end{equation}
is well-defined. Next, we argue that, actually, it holds that~$s_m = s$. Otherwise, i.e., in case~$s_m<s$,  by choice of~$s_m$ and since~$s$ is chosen as largest index in the range~$i+1, \dots, k-1$ that has minimum $\funsymbd$-value, we must have~$s_m \le i$, and thus
\[t_0 < s_m \le i\leq i'< s <r .\]
Therefore, the index~$i'$ belongs to the index set used to define $t_m$ according to~\eqref{eq:define-t-h}, while the values~$\func{i'+1}, \dots, \func{r-1}$ are all strictly smaller than~$\func{i'}$. The latter assertion follows for the indices in the considered range that are strictly smaller, equal, and strictly larger than~$s$ from choice of~$i'$, from~\eqref{eq:s-t-greater-than-i-k}, and from choice of~$r$, respectively. It follows that~$t_m \leq i'$, hence~$s_{m+1}$ exists and is equal to~$s$ by minimality of~$\fund{s}$ and by choice of~$s_{m+1}$ in the range~${t_m +1, \dots, r-1}$, which contains~$s$ by~${i^{\prime} < s <r}$. But, by definition of~$m$, we have~${s < s_{m+1}}$, a contradiction. Consequently, we have~${s_m = s}$. 

Observe that we have
\begin{equation}\label{eq:dcp-le-gamma-r-delta-s}
\ee{p} < \func{i}  - \fund{j}
\leq \func{i^{\prime}}  - \fund{s}
\leq \func{r}  - \fund{s},
\end{equation}
where the inequalities hold, from left to right, by assumption of the claim, by~\eqref{eq:s-t-greater-than-i-k}, and by choice of~$r$.

In case~$m=1$, we are done because then we have by construction
\begin{equation}\label{eq:gamma-r-delta-s-le-max-u-s-k}
\func{r}  - \fund{s}  = \max{\intu{s}{r}} \le  \max{\intu{s}{k}},
\end{equation}
hence~$\ee{p}$ is indeed strictly smaller than the right endpoint of the interval~$\intu{s}{k}$. 

So, from now one, we can assume~$m>1$. Then~$s_{m-1}$ and~$t_{m-1}$ are defined, and the upper bound of the interval~$\intu{s}{r}$ is set equal to~$\func{t_{m-1}}-\delta_s$ by~\eqref{eq:step-definition-terminal}. Consequently, in case~$\func{i'}\leq\func{t_{m-1}}$, both of \eqref{eq:dcp-le-gamma-r-delta-s} and~\eqref{eq:gamma-r-delta-s-le-max-u-s-k} hold true with~$\func{r}$ replaced by~$\func{t_{m-1}}$, and we are done by essentially the same argument as in case~$m=1$. 

We conclude the proof of the claim assertion by demonstrating the inequality ${\func{i'}\leq\func{t_{m-1}}}$. The index~$s$ is chosen in the range~$i+1, \dots, k-1$ as largest index with minimum $\funsymbd$-value. The latter range  contains the range~$i+1, \dots, r-1$  because we have~$i < s < r \le k$.  The index~$s_{m-1}$ differs from~${s=s_m}$ and is chosen as the largest index with minimum $\funsymbd$-value among indices that are less than or equal to~$r-1$, hence~${s_{m-1} \le i}$. By~$i \le i^{\prime} < s <r$, the index $i'$ belongs to the range~$s_{m-1}, \dots, r-1$, from which~$t_{m-1}$ is chosen as largest index with maximum $\funsymbd$-value according to~\eqref{eq:define-t-h}, hence we obtain that~${\func{i'}\leq\func{t_{m-1}}}$. 
\end{proof}

\begin{claimbbb}\label{claim:coverset-coincides-coverall}
Let $Q$ be a nonempty finite set of rationals, and let~$p$ be a nonrational real in~$[0,1]$. In case $p>\max Q$, it holds that
\begin{equation}\label{eq:coverset-coincides-coverall}
\coverset{Q}{\ee{p}} = \coverall{Q}{\ee{p}}.
\end{equation}
\end{claimbbb}

\begin{proof}
The inequality~$\coverset{Q}{\ee{p}} \leq \coverall{Q}{\ee{p}}$ is immediate by definition of~$\coverall{Q}{\ee{p}}$. We show the reverse inequality~$\coverset{Q}{\ee{p}} \geq \coverall{Q}{\ee{p}}$ by induction on the size of~$Q$.

In the base case, let~$Q$ be empty or a singleton set. The induction claim holds in case~$Q$ is empty because then~$Q$ is its only subset as well as in case~$Q$ is a singleton because then~$\coverallfun{Q}$ is equal to the maximum of~$\coversetfun{Q}$ and~$\coversetfun{\emptyset}$, where the latter function is identically~$0$. 

In the induction step, let~$Q$ be of size at least~$2$. For a proof by contradiction, assume that the induction claim does not hold true for~$Q$, i.e., that there exist a subset $H$ of~$Q$ such that 
\begin{equation}\label{eq:Q-H-to-contradict}
\coverset{Q}{\ee{p}} < \coverset{H}{\ee{p}}.
\end{equation}

Then we have the following chain of inequalities
\begin{equation} \label{eq:circle-ineuality}
\coverset{Q}{\ee{p}} \geq \coverset{Q\setminus \{\max Q\}}{\ee{p}} \geq \coverset{H\setminus\{\max Q\}}{\ee{p}} \geq \coverset{H}{\ee{p}} - 1 \ge \coverset{Q}{\ee{p}},
\end{equation}
where the first and the third inequalities hold true by Claim~\ref{claim:k<k<k+1}, the second one holds by the induction hypothesis for the set~$Q\setminus \{\max Q\}$, and the fourth one by~\eqref{eq:Q-H-to-contradict}. The first and the last values in the chain~\eqref{eq:circle-ineuality} are identical, and thus the chain remains true when we replace all inequality symbols by equality symbols, i.e., we obtain
\begin{equation} \label{eq:circle-equality}
\coverset{Q}{\ee{p}} = \coverset{Q\setminus \{\max Q\}}{\ee{p}} = \coverset{H\setminus\{\max Q\}}{\ee{p}} = \coverset{H}{\ee{p}} - 1 = \coverset{Q}{\ee{p}}.
\end{equation}
Since~$\coverset{H}{\ee{p}}$ is strictly larger than~$\coverset{H\setminus\{\max Q\}}{\ee{p}}$, the set~$H$ must contain~$\max Q$, hence~$\max Q$ and~$\max H$ coincide, and~$H$ has size at least two.

Now, let~$Q=\{q_0, \dots, q_n\}$, where~$q_0 < \cdots < q_n$, and let~$H = \{q_{\zindex{0}}, \dots, q_{\zindex{n_H}}\}$, where~$\zindex{0} <  \dots < \zindex{n_H}$. Furthermore, let~$Q_i=\{q_0, \dots, q_i\}$ for~$i=0, \dots, n$, and let~$H_i=\{q_{\zindex{0}}, \dots, q_{\zindex{i}}\}$ for~$i=0, \dots, n_H$.
So, the set~$Q$ has size~$n+1$, its subset~$H$ has size~$n_H+1$, and the function~$\zindexfunsymb$ transforms indices with respect to~$H$ into indices with respect to~$Q$. For example, since the maxima of~$Q$ and~$H$ coincide, we have~$\zindex{n_H}=n$. 

In what follows, we consider the construction of~$M(H)$. The index stairs that occur in this construction contain indices with respect to~$H$, i.e., for example, the index~$t_0$ refers to~$q_{\zindex{t_0}}$. A similar remark holds for the intervals that occur in the construction of~$M(H)$, i.e., for such an interval~$\intu{s}{t}$, we have
\[\max \intu{s}{t} = 
\func{\zindex{t}}-\fund{\zindex{s}} = \funsymbc(q_{\zindex{t}}) - \funsymbd(q_{\zindex{s}}).\]
However, as usual, for a given index~$i$, we write~$\func{i}$  for~$\funsymbc(q_i)$ and~$\fund{i}$  for~$\funsymbd(q_i)$. 

For every interval of the form~$\intu{s}{t}$ that occurs in some step of the construction of~$M(H)$, the left endpoint~$\ee{q_{\zindex{s}}}$ of this interval is strictly smaller than~$\ee{p}$ by assumption of the claim, hence, for every such interval, it holds that
\begin{equation}\label{eq:d-c-max-u-less-than-d-c-p}
\ee{p} \in  \intu{s}{t}  \makebox[10em]{ if and only if}
\ee{p} \le \max    \intu{s}{t}.   
\end{equation}

Let~$(t_0,s_1,t_1, \dots,s_l)$ be the index stair of step~$n_H$ of the construction of~$M(H)$, i.e., of the last step, and recall that these indices are chosen with respect to~$H$, e.g., the index~$t_0$ stands for~$q_{\zindex{t_0}}$. By the third equality in~\eqref{eq:circle-equality}, for some index $h\in\{1,\dots,l\}$, the interval~$\intv{s_h}{n_H}$ added during this step contains~$\ee{p}$, that is, 
\begin{equation} \label{eq:intvTilde-containing-p}
\ee{p} \in \intv{s_h}{n_H} = \intu{s_h}{n_H} \setminus \intu{s_h}{n_H-1}.
\end{equation}
By the explicit descriptions for the left and right endpoint of~$\intv{s_h}{n_H}$ according to Claim~\ref{V=U-U}, we obtain
\begin{eqnarray}
\label{eq:p-bounds-for-m=1}
\func{\zindex{t_1}}-\fund{\zindex{s_1}}< \, \ee{p}& \, 
\le \func{\zindex{n_H}}-\fund{\zindex{s_1}} &\text{ if } h=1,
\\
\label{eq:p-bounds-for-m>1}
\func{\zindex{t_h}}-\fund{\zindex{s_h}}< \, \ee{p}& \,
\le \func{\zindex{t_{h-1}}}-\fund{\zindex{s_h}}
\le \func{\zindex{n_H}}-\fund{\zindex{s_h}} 
&\text{ if } h>1.
\end{eqnarray}
So, in the last step of the construction of~$M(H)$, the real~$\ee{p}$ is covered via the expansion of the interval with index~$s_h$. We argue next that, in the construction of~$M(H)$, the last step before step~$n_H$, in which~$\ee{p}$ is covered by the expansion of some interval, must be not larger than~$s_h$, i.e., we show
\begin{equation}\label{eq:first-equality-for-H}
\coverset{H_{s_h}}{\ee{p}} = \coverset{H\setminus\{\max H\}}{\ee{p}}.
\end{equation}

For a proof by contradiction, assume that this equation is false. Then there is a step~$x$ of the construction of~$M(H)$ with index stair~$(t_0',s'_1,t'_1,\dots,s'_{l'},t'_{l'})$ and some index~$i$ in~$\{1, \dots, l'\}$ such that
\begin{equation}\label{eq:stage-y-index-stair}
s_h < x  < n_H \makebox[5em]{ and }
\ee{p} \in \intv{s'_{i}}{x} = \intu{s'_{i}}{x} \setminus \intu{s'_{i}}{x-1}.
\end{equation}
Observe that the indices in this index stair are indices with respect to the set~$H_x $ but coincide with indices with respect to the set~$H$ because~$H_x$ is an initial segment of~$H$ in the sense that~$H_x$ contains the least~$x+1$ members of~$H$. In particular, the index transformation via the function~$\zindexfunsymb$ works also for the indices in this index stair, for example, the index~$t_0'$ refers to~$z(t_0')$.

We have~$s'_{i} \le x$ because, otherwise, the interval~$\intv{s'_{i}}{x}$ would be empty by~Claim~\ref{claim:before}. Furthermore, the indices~$s_h$ and~$s'_{i}$ must be distinct because~$\ee{p}$ is contained in both of the intervals~$\intv{s_h}{n_H}$ and~$\intv{s'_{i}}{x}$, 
while the former interval is disjoint from the interval~$\intv{s_h}{x}$ by~$\intv{s_h}{x} \subseteq \intu{s_h}{x} \subseteq \intu{s_h}{n_H-1}$ and~$\intu{s_h}{n_H-1}\cap\intv{s_h}{n_H} = \emptyset$.

Next, we argue that
\begin{equation}\label{eq:gamma-chain-t-j-prime-to-n-h}
\func{\zindex{t'_i}}
\le \func{\zindex{x}}
\le \func{\zindex{t_h}}
\le \func{\zindex{n_H}}
<\func{\zindex{t_0}}
\makebox[5em]{ and }
\func{\zindex{x}} < \func{\zindex{t_{h-1}}}.
\end{equation}
In the chain on the left, the last two inequalities hold by~$t_0 < t_h < n_H$ and by definition of~$t_0$. The first inequality holds because, otherwise, in step~$x$, the index~$t_i' > t_0'$ would have been chosen in place of~$t_0'$. The second inequality holds by choice of~$t_h$ as largest index in the range~$s_h, \dots, n_H-1$  that has maximum \mbox{$\funsymbc$-value} and because this range contains~$x$. 
From the latter inequality then follows the single inequality on the right since~$\func{\zindex{t_{h}}} < \func{\zindex{t_{h-1}}}$ holds by definition of index stair. From~\eqref{eq:gamma-chain-t-j-prime-to-n-h}, we now obtain
\begin{equation}\label{eq:t-i-minus-one-u-prime-plus-s}
t_{h-1} \le t_0' \makebox[8em]{ and }
\fund{\zindex{s_h}} \le \fund{z(s'_1)} \le \fund{z(s'_i)}.
\end{equation}
Here, the first inequality is implied by the right part of~\eqref{eq:gamma-chain-t-j-prime-to-n-h} and choice of~$t_0'$. The last inequality holds by definition of index stair. The remaining inequality holds because~$s_h$ and~$s'_1$ are chosen as largest indices with minimum $\funsymbd$-value in the ranges~$t_{h-1}+1, \dots, n_H-1$ and~$t_0'+1, \dots, x-1$, respectively, where the latter range is a subset of the former one by the just demonstrated first inequality and since~$x$ is in~$H$.

Now, we obtain as a contradiction to~\eqref{eq:intvTilde-containing-p} that~$\ee{p}$ is in~$\intu{s_h}{n_H-1}$ since we have
\[
\ee{p} 
\le \max \intu{s'_i}{x}
\le \func{\zindex{x}} - \fund{\zindex{s'_i}}
\le \func{\zindex{t_h}} - \fund{\zindex{s_h}}
= \max \intu{s_h}{n_H-1}. 
\]
Here, the first inequality holds because~$\ee{p}$ is in~$\intu{s'_i}{x}$ by choice of~$i$ and~$x$. The second inequality holds because, by construction, $\max \intu{s'_i}{x}$ is equal to~$\func{\zindex{x}} - \fund{\zindex{s'_i}}$ in case~$i=1$ and is equal to~$\func{\zindex{t'_{i-1}}} - \fund{\zindex{s'_i}}$ in case~$i>1$, where~$\func{\zindex{t'_{i-1}}} \le \func{\zindex{x}}$ by~\eqref{eq:stair-values} since~$t'_{i-1}$ lies in the index stair of step~$x$ and~$i-1\geq 1$.  The third inequality holds by~\eqref{eq:gamma-chain-t-j-prime-to-n-h} and~\eqref{eq:t-i-minus-one-u-prime-plus-s}, and the final equality holds by Claim~\ref{claim:previous-step}. This concludes the proof of~\eqref{eq:first-equality-for-H}. 

By~\eqref{eq:first-equality-for-H}, during the steps~$s_h+1, \dots, n_H-1$, none of the expansions of any interval covers~$\ee{p}$. Now, let~$y$ be the minimum index in the range~$t_{h-1}, \dots, s_h$ such that, during the steps~$y+1, \dots, s_h$, none of the expansions of any interval covers~$\ee{p}$, i.e.,  
\begin{equation}\label{eq:definition-of-y<=s_i}
y =\min\{k \colon t_{h-1}\leq k \leq s_h \text{ and }
\coverset{H_k}{\ee{p}} = \coverset{H_{s_h}}{\ee{p}}\}.
\end{equation}
Note that~$y$ is an index with respect to the set~$H$. We demonstrate that the index~$y$ satisfies
\begin{equation}\label{eq:property-for-y}
\ee{p} \le \func{\zindex{y}} - \fund{\zindex{s_h}}.
\end{equation}
For further use, note that inequality~\eqref{eq:property-for-y} implies that~$y$ and~$s_h$ are distinct because,  otherwise, since we have~$\max Q  < p$, we would obtain the contradiction:
\[
\ee{p}
\le \func{\zindex{y}} - \fund{\zindex{s_h}}
= \func{\zindex{s_h}} - \fund{\zindex{s_h}}
= \ee{q_{\zindex{s_h}}}.
\]

Now, we show~\eqref{eq:property-for-y}. Assuming~$y = t_{h-1}$, the inequality is immediate by~\eqref{eq:p-bounds-for-m=1} and choice of~$t_0$ in case~$h=1$ and by~\eqref{eq:p-bounds-for-m>1} in case~$h>1$. So, in the remainder of the proof of~\eqref{eq:property-for-y}, we can assume~$t_{h-1}<y$. 

By choice of~$y$, we have~$\coverset{H_{y-1}}{p} \neq \coverset{H_{y}}{p}$, that implies~$\coverset{H_{y-1}}{p} < \coverset{H_{y}}{p}$ by Claim~\ref{claim:previous-step}. Consequently, for the index stair $(t_0'', s''_1,t''_1,\dots,s''_{l''},t''_{l''})$ of step $y$ of the construction of the test~$\miller{H}$, there exists an index~$j\in\{1,\dots,l''\}$ such that~$\ee{p}$ is in~$\intv{s''_{j}}{y}$. Thus, in particular, it holds that
\begin{equation}\label{eq:d-c-p-le-s-double-prime-j}
\ee{p}
\le \max \intu{s''_{j}}{y}
\le 
\func{\zindex{y}} - \fund{\zindex{s''_{j}}}
\end{equation}
because, by construction, the value~$\max \intu{s''_{j}}{y}$ is equal to~$\func{\zindex{y}} - \fund{\zindex{s''_{j}}}$ in case~$j=1$ and is equal to~$\func{\zindex{t''_{j-1}}} - \fund{\zindex{s''_{j}}}$ in case~$h>1$, where~$\func{\zindex{t''_{j-1}}}\le\func{\zindex{y}}$. By~\eqref{eq:d-c-p-le-s-double-prime-j}, it is then immediate that, in order to demonstrate~\eqref{eq:property-for-y}, it suffices to show that
\begin{equation}\label{eq:s-h-le-s-double-prime-j}
\fund{\zindex{s_h}} \le \fund{\zindex{s''_{j}}}. 
\end{equation}
The latter inequality follows in turn if we can show that
\begin{equation}\label{eq:t-t-hminuseins-lt-y-lt-sh-lt-nh}
t_{h-1} \le t_0'' \le t_{j-1}'' < y \le s_h < n_H,
\end{equation}
because the indices~$s_h$ and~$s''_j$ are chosen as largest indices with minimum $\funsymbd$-value in the ranges~$t_{h-1}+1, \dots, n_H-1$ and~$t''_{j-1}+1, \dots, y-1$, respectively, where the latter range is a subset of the former. 

We conclude the proof of~\eqref{eq:s-h-le-s-double-prime-j}, and thus also of~\eqref{eq:property-for-y}, by showing~\eqref{eq:t-t-hminuseins-lt-y-lt-sh-lt-nh}.
The second to last inequality holds by choice of~$y$, and all other inequalities hold by definition of index stair, except for the first one. Concerning the latter, by our assumption~$t_{h-1}< y$, by $y < n_H$, and by choice of~$t_{h-1}$, we obtain that~$\func{\zindex{y}}  < \func{\zindex{t_{h-1}}}$, which implies that~$t_{h-1} \le t_0''$ by choice of~$t_0''$.

Now, we can conclude the proof of the claim. For the set~$Q$ and the indices ${\zindex{y}< \zindex{s_h} <  \zindex{n_H}=n}$, by~\eqref{eq:p-bounds-for-m=1}, \eqref{eq:p-bounds-for-m>1}, and~\eqref{eq:property-for-y}, all assumptions of Claim~\ref{i<j<k-coverage} are satisfied, hence the claim yields that
\begin{equation} \label{covering-y<covering<n}
\coverset{Q_{\zindex{s_h}}}{\ee{p}}<\coverset{Q}{\ee{p}}.
\end{equation}
So we obtain the contradiction
\[
\coverset{Q}{\ee{p}}
= \coverset{H\setminus\{\max Q\}}{\ee{p}}  
= \coverset{H_{s_h}}{\ee{p}} 
\le \coverset{Q_{\zindex{s_h}}}{\ee{p}} 
< \coverset{Q}{\ee{p}},
\]
where the relations follow, from left to right, 
by~\eqref{eq:circle-equality},
by~\eqref{eq:first-equality-for-H}, 
by the induction hypothesis for the set~$Q_{\zindex{s_h}}$,
and by~\eqref{covering-y<covering<n}. 
\end{proof}

\begin{claimbbb}\label{after-p-every-new-covering-is-last}
Let~$Q= \{q_0 <  \cdots < q_n\}$ be a subset of the domain of~$g$, and for~$z= 0, \dots, n$, let~$Q_z= \{q_0, \dots, q_z\}$. Let~$p$ be  a nonrational real such that, for some index~$x$ in~$\{1, \dots, n\}$, it holds that $p\in[0,q_x]$ and
\begin{equation}\label{eq:claim-q-x-minus-1-neq-k-q-x}
\coverset{Q_{x-1}}{\ee{p}}\neq \coverset{Q_x}{\ee{p}}.
\end{equation}
Then it holds that 
\[
\coverset{Q_{x}}{\ee{p}} = \coverset{Q_{x+1}}{\ee{p}} = \dots = \coverset{Q_{n}}{\ee{p}}.
\]
\end{claimbbb}

\begin{proof}
We denote the intervals considered in the construction of the test~$\miller{Q}$ by~$\intu{i}{j}$, as usual. Again, we can argue that the construction of a test of the form~$\miller{Q_z}$ where~$z \le n$ is essentially identical to an initial part of the construction of~$\miller{Q}$, and that accordingly such a test~$\miller{Q_z}$ coincides with $(\intu{0}{z}, \dots, \intu{z}{z})$.

Let~$(t_0 ,s_1, t_1, \dots,s_l, t_l)$ be the index stair of step~$x$ in the construction of~$\miller{Q_n}$. By~\eqref{eq:claim-q-x-minus-1-neq-k-q-x}, there is an index~$h$ in~$\{1, \dots, l\}$ such that~$\ee{p}$ is in~$\intv{s_h}{x}$, hence 
\begin{equation}\label{eq:inequality-for-s-h}
\ee{q_{s_h}} = \min\intu{s_h}{x-1} \le \max\intu{s_h}{x-1}
= \func{t_h}-\fund{s_h} <\ee{p}\le \func{x}-\fund{s_h}.
\end{equation}
Here, the two equalities hold by definition of the interval and by Claim~\ref{claim:previous-step}, respectively. The strict inequality holds because~$\ee{p}$ is assumed not to be in~$\intu{s_h}{x-1}$. The last inequality holds because~$\ee{p}$ is assumed to be in~$\intu{s_h}{x-1}$, while, by construction, the right endpoint of the latter interval is equal to~$\func{x}-\fund{s_h}$ in case~$h=1$ and is equal to~$\func{t_h} - \fund{s_h}$ with~$\func{t_h} \leq \func{x}$ otherwise.

For a proof by contradiction, we assume that the conclusion of the claim is false. So we can fix an index $y\in\{x+1,\dots,n\}$ such that 
\[
\coverset{Q_{x}}{\ee{p}} = \coverset{Q_{y-1}}{\ee{p}} < \coverset{Q_{y}}{\ee{p}}.
\]
Let~$(t_0',s'_1,t'_1,\dots,s'_{l'},t'_{l'})$ be the index stair of step $y$ of the construction of~$\miller{Q_n}$. By essentially the same argument as in the case of~\eqref{eq:inequality-for-s-h}, we can fix an index~$i$ in~$\{1, \dots, l'\}$ such that~$\ee{p}$ is in~$\intv{s'_{i}}{y}$, and therefore, that
\begin{equation}\label{eq:inequality-for-s-i}
\ee{q_{s'_{i}}} = \min\intu{s'_i}{y-1} \le \max\intu{s'_i}{y-1} 
=  \func{t'_{i}}-\fund{s'_{i}}<\ee{p} \le \func{y}-\fund{s'_{i}}.
\end{equation}
By assumption, the real~$p$ is in~$[0,q_x]$, and together with~\eqref{eq:inequality-for-s-h} and~\eqref{eq:inequality-for-s-i}, we obtain $q_{s_h}<p \le q_x$ and $q_{s'_{i}}<p \le q_x$. Consequently, we have
\begin{equation}\label{eq:s,s'<x}
s_h<x \makebox[8em]{and} s'_{i}<x
\end{equation}
(where the left inequality also follows from definition of index stair). In particular, we have~$t_0' < x$, which implies by~$x<y$ and choice of~$t_0'$ that
\begin{equation}\label{eq:x<=y}
\func{y}< \func{x}.
\end{equation}
In order to derive the desired contradiction, we distinguish the three cases that are left open by~\eqref{eq:s,s'<x} for the relative sizes of the indices~$s_h$, $s'_{i}$, and~$x$. 

\underline{Case 1}: $s_h < s'_{i} < x$. Since~$s_h$ is chosen in the range~$t_{h-1}+1, \dots, x-1$ as largest index with minimum $\funsymbd$-value, we obtain by case assumption that
\begin{equation}\label{eq:s<s'}
\fund{s_h}<\fund{s'_{i}}.
\end{equation}
Furthermore, it holds that
\begin{equation}\label{eq:u-s-h-t-prime-i-minus-1-s-prime}
t_0 < s_h \le t'_{i-1}< s'_{i} < x < y.
\end{equation}
Here, the first and the third inequalities hold by definition of index stair. The two last inequalities hold by~\eqref{eq:s,s'<x} and by choice of~$y$, respectively. The remaining second inequality holds because, otherwise, i.e., in case~$t'_{i-1}< s_h$, the range~${t'_{i-1}+1, \dots, y-1}$, from which~$s'_{i}$ is chosen as largest index with minimum $\funsymbd$-value, would contain~$s_h$, which contradicts~\eqref{eq:s<s'}. 

Now, we obtain a contradiction, which concludes Case~1. Due to~$t_0 < t'_{i-1} <x$ and definition of~$t_0$, we have~$\func{t'_{i-1}}\leq\func{x}$. The latter inequality contradicts the fact that~$t'_{i-1}$ is chosen in the range~$s'_{i-1}+1, \dots, y-1$ as largest index with maximum $\funsymbc$-value, where the latter range contains~$x$ by~\eqref{eq:u-s-h-t-prime-i-minus-1-s-prime} and~$s'_{i-1} < s'_{i}$.

\underline{Case 2}: $s_h = s'_{i} < x$. In this case, we have
\[
\ee{p} \in \intv{s_h}{x} \makebox[4em]{and} \ee{p} \in \intv{s'_{i}}{y} = \intv{s_h}{y},\makebox[7em]{and thus,} \ee{p} \in \intv{s_h}{x}\cap \intv{s_h}{y},
\]
which cannot hold since~$\intv{s_h}{x}$ and~$\intv{s_h}{y}$ are disjoint by Claim~\ref{V=U-U}.

\underline{Case 3}: $s'_{i} < s_h < x$. In this case, we have
\begin{equation}\label{eq:s'<s-and-next-t'>x}
\fund{s'_{i}}<\fund{s_h} \makebox[6em]{ and}
\func{x} \le \func{t'_{i}}.
\end{equation}
Here, the first inequality holds since~$s'_{i}$ is chosen as largest index with minimum $\funsymbd$-value from a range that, by case assumption, contains~$s_h$. The second inequality holds since~$t'_{i}$ is chosen in the range $s'_{i}+1, \dots, y-1$ as largest index with maximum $\funsymbc$-value, where this range contains~$x$ by case assumption and~$x<y$. 

Now, we obtain a contradiction, which concludes Case~3, since we have
\begin{equation}
\ee{p} \leq \func{x} - \fund{s_h} < \func{t'_{i}} - \fund{s'_{i}}  < \ee{p},
\end{equation}
where the inequalities hold, from left to right, by~\eqref{eq:inequality-for-s-h}, by~\eqref{eq:s'<s-and-next-t'>x}, and by~\eqref{eq:inequality-for-s-i}. 

So we obtain in all three cases a contradiction, which concludes the proof of the claim.
\end{proof}

\subsection*{The proof of Claim~\ref{general-bound-for-coverall}}
Let~$Q= \{q_0 <  \cdots < q_n\}$ where $q_n<1$ be a subset of the domain of~$g$.
For~$z= 0, \dots, n$, let~$Q_z= \{q_0, \dots, q_z\}$, and let~$p$ be an arbitrary nonrational real in~$[0,1]$.
In order to demonstrate Claim~\ref{general-bound-for-coverall}, it suffices to show
\begin{equation}\label{eq:inequality-of-claim-general-bound-for-coverall}
\coverall{Q}{\ee{p}} \le \coverset{Q}{\ee{p}} +1.
\end{equation}
Since~$p$ was chosen as an arbitrary nonrational real in~$[0,1]$, this
easily implies the assertion of Claim~\ref{general-bound-for-coverall}, i.e., that~$\coverall{Q}{p'} \le \coverset{Q}{p'} + 1$ for all nonrational~$p'$ in~$[0, \eesymb]$.

By construction, for all subsets~$H$ of~$Q$, all intervals in the test~$\miller{H}$ have left endpoints of the form~$\funsymbc(q_i)-\funsymbd(q_i) = \ee{q_i}$. Consequently, in case~$p<q_0$, none of such intervals contains~$\ee{p}$, hence~$\coverall{Q_n}{\ee{p}} = 0$, and we are done. 

So, from now on, we can assume~$q_0 < p$. Then, among~$q_0, \dots, q_n$, there is a maximum value that is smaller than~$p$, and we let 
\begin{equation}
j=\max\big\{i\in\{0,\dots,n\}:q_i<p\big\}
\end{equation}
be the corresponding index. It then holds that
\begin{equation}\label{eq:coverset=coverall-p-and-j<...<n}
\coverall{Q_j}{\ee{p}} = \coverset{Q_j}{\ee{p}}
\leq\coverset{Q_{j+1}}{\ee{p}}\leq\dots \le \coverset{Q_{n-1}}{\ee{p}} \le \coverset{Q}{\ee{p}},
\end{equation}
where the equality is implied by choice of~$j$ and Claim~\ref{claim:coverset-coincides-coverall}, and the inequalities hold by Claim~\ref{claim:k<k<k+1}. 

Fix some subset~$H$ of~$Q$ that realizes the value~$\coverall{Q}{\ee{p}}$ in the sense that 
\begin{equation}\label{eq:choice-of-H}
\coverall{Q}{\ee{p}}=\coverset{H}{\ee{p}}.
\end{equation}
Next, we show that, for the set~$H$, we have
\begin{equation}\label{eq:inequality-h-cap-q-j-plus-1}
\coverset{H}{\ee{p}} \le \coverset{H\cap Q_j}{\ee{p}} +1.
\end{equation}
In case~$\coverset{H}{\ee{p}} \le \coverset{H\cap Q_j}{\ee{p}}$, we are done. Otherwise, let~$x$ be the least index in the range~$j+1, \dots, n$ such that $\coverset{H\cap Q_{x}}{\ee{p}}$ differs from $\coverset{H\cap Q_{x-1}}{\ee{p}}$. 
Then~\eqref{eq:inequality-h-cap-q-j-plus-1} follows from
\[
\coverset{H\cap Q_j}{\ee{p}} +1
= \coverset{H\cap Q_{x-1}}{\ee{p}} + 1 
= \coverset{H\cap Q_{x}}{\ee{p}}
= \coverset{H\cap Q}{\ee{p}} = \coverset{H}{\ee{p}},
\]
where the equalities hold, from left to right, by choice of~$x$, by Claim~\ref{claim:k<k<k+1}, by Claim~\ref{after-p-every-new-covering-is-last}, and since~$H$ is a subset of~$Q$.

Now, we have
\[
\coverall{Q}{\ee{p}}
= \coverset{H}{\ee{p}}
\le \coverset{H\cap Q_j}{\ee{p}} +1
\le \coverall{Q_j}{\ee{p}} +1
\le \coverset{Q}{\ee{p}} +1,
\]
where the relations hold, from left to right, by choice of~$H$, by~\eqref{eq:inequality-h-cap-q-j-plus-1}, because ${H\cap Q_j}$ is a subset of~$Q_j$, and by \eqref{eq:coverset=coverall-p-and-j<...<n}.

This concludes the proof of~\eqref{eq:inequality-of-claim-general-bound-for-coverall} and thus also of Claim~\ref{general-bound-for-coverall} and, finally, of~\eqref{eq:blp-generalized}. 

\medskip

\subsection{The left limit is unique}\label{proof:unique}
At that point, we have demonstrated that, for every nondecreasing translation function $g$ from a Martin-Löf random real~$\beta$ to a real $\alpha$, the left limit $\lim\limits_{q\nearrow\beta}\frac{\alpha - g(q)}{\beta - q}$ exists and is finite. It remains to show that this left limit does not depend on the choice of the translation function from~$\beta$ to~$\alpha$. For a proof by contradiction, assume that there exist two translation functions $f$ and $g$ from~$\beta$ to~$\alpha$ such that the values $\lim\limits_{q\nearrow\beta}\frac{\alpha - f(q)}{\beta - q}$ and $\lim\limits_{q\nearrow\beta}\frac{\alpha - g(q)}{\beta - q}$ differ. By symmetry, without loss of generality, we can then pick rationals~$c$ and~$d$ such that
\begin{equation}\label{eq:uniqueness-c-and-d}
\lim\limits_{q\nearrow \beta}\frac{\alpha - g(q)}{\beta - q} < c < d <
\lim\limits_{q\nearrow \beta}\frac{\alpha - f(q)}{\beta - q}.
\end{equation}
By~\eqref{eq:uniqueness-c-and-d}, for every rational~$q < \beta$ that is close enough to~$\beta$, it holds that
\[
\frac{\alpha - g(q)}{\beta - q} < c
\makebox[5em]{and} 
d < \frac{\alpha - f(q)}{\beta - q}.    
\]

Fix some rational~$p < \beta$ such that the two latter inequalities are both true for all rationals~$q$ in the interval~$[p, \beta)$. We then have for all such~$q$ 
\begin{equation}\label{eq:towards-upper-bound-for-e-beta}
0 < d-c <\frac{\big(\alpha - f(q)\big) - \big(\alpha - g(q)\big)}{\beta - q} = \frac{g(q)-f(q)}{\beta - q},
\end{equation}
and consequently, letting~$e = d-c$,
\begin{equation}\label{eq:two-policemen}
    \ee{q} < \ee{\beta} < \ee{q} + g(q)-f(q)
\quad \text{ for all rationals~$q$ in~$[p, \beta)$}, 
\end{equation}
where the lower bound is immediate by~$q < \beta$, and the upper bound follows 
by multiplying the first and the last terms in~\eqref{eq:towards-upper-bound-for-e-beta} by~$\beta-q$ and rearranging. Let
\[
D = \{ q \in [0,1]\colon \text{$f$ and~$g$ are both defined on~$q$, and~$f(q) < g(q)$} \}.
\]
For every~$q$ in~$D$, define the intervals
\[
\intiunique{q}= [f(q), g(q)]
\makebox[5em]{ and }
\intuunique{q}= [\ee{q}, \ee{q} + g(q) - f(q)]. 
\]

Fix some effective enumeration~$q_0, q_1, \dots$ of~$D$. We define inductively a subset~$S$ of the natural numbers and let~$S_n$ be the intersection of~$S$ with~$\{0,\dots, n\}$. Let~$0$ be in~$S$, and for~$n>0$, assuming that~$S_n$ has already been defined, let
\[
n+1 \in S \makebox[7em]{ if and only if, }
\text{ for all~$i$ in~$S_n$, the intervals~$\intiunique{q_{i}}$ and~$\intiunique{q_{n+1}}$ are disjoint}. 
\]
The intervals~$\intuunique{q_n}$, where~$n$ is in~$S$, form a Solovay test. First, these intervals can be effectively enumerated since~$S$ is computable by construction. Second, the sum of the lengths of these intervals is at most~$1$ because, for every~$q$ in~$D$, the intervals~$\intuunique{q}$ and~$\intiunique{q}$ have the same length by definition, while the intervals~$\intiunique{q_{n}}$, where~$n$ is in~$S$, are mutually disjoint by definition of~$S$.

\medskip

So, in order to obtain the desired contradiction, it suffices to show that the Martin-Löf random real~$\ee{\beta}$ is covered by the Solovay test just defined, i.e., that there are infinitely many~$i$ in~$S$ such that the interval~$\intuunique{q_{i}}$ contains~$\ee{\beta}$. By definition of these intervals and~\eqref{eq:two-policemen}, here it suffices in turn to show that there are infinitely many~$i$ such that~$i$ is in~$S$ and~$q_i$ is in~$[p,\beta)$. To this end, we fix some arbitrary natural number~$n$ and show that there is such~$i>n$. 

Since the values~$f(q)$ converge from below to~$\alpha$ when~$q$ tends from below to~$\beta$, we can fix an index~$i_0 > n$ such that~$q_{i_0} \in  [p, \beta)$, and in addition, we have
\begin{equation}\label{eq:uniqueness-varepsilon-lt-q}
\text{(i) } g(p) < f(q_{i_0})
\makebox[3.5em]{ and }
\text{(ii) }g(q_i) < f(q_{i_0}) \;
\text{ for all~$i$ in~$S_n$ where~$q_i < \beta$.}
\end{equation}
In case~$i_0$ is selected by~$S$, we are done. Otherwise, there must be some~$i_1 < i_0$ in~$S$ such that~$\intiunique{q_{i_1}}$ has a nonempty intersection with~$\intiunique{q_{i_0}}$. We fix such an index~$i_1$ and conclude the proof by showing that we must have~$n<i_1$ and~$q_{i_1} \in [p, \beta)$.

In order to prove the latter, we show for~$q$ in~$D$ that, in case~$q < p$ and in case~$q \geq \beta$, the intervals~$\intiunique{q}$ and~$\intiunique{q_{i_0}}$ are disjoint. In the former case, by monotonicity of~$g$, the right endpoint~$g(q)$ of the interval~$\intiunique{q}$ is strictly smaller than the left endpoint~$f(q_{i_0}) > g(p)$ of~$\intiunique{q_{i_0}}$. In the latter case, the left endpoint~$f(q)$ of~$\intiunique{q}$ is at least as large as~$\alpha$, and thus is strictly larger than the right endpoint of~$\intiunique{q_{i_0}}$ since $f$ maps $[0,\beta)$ in $[0,\alpha)$ as a translation function. Otherwise, i.e., in case~$f(q) < \alpha$, since the values~$f(q')$ converge from below to~$\alpha$ when~$q'$ tends from below to~$\beta$, there would be~$q'< \beta$ where~$f(q) < f(q')$, contradicting the monotonicity of~$f$.

It remains to show that~$n < i_1$, i.e., that~$i_1$ is not in~$S_n$. But, for any index~$i$ in~$S_n$, the intervals~$\intiunique{q_{i}}$ and~$\intiunique{q_{i_0}}$ are disjoint as follows in case~$q_i \geq \beta$ from the discussion in the preceding paragraph and follows in case~$q_i < \beta$ from~(ii) in~\eqref{eq:uniqueness-varepsilon-lt-q}.

This concludes the proof of uniqueness of the left limit as well as the whole proof of Theorem~\ref{theorem:BLP-generalized}.
\end{proof}

\section{Conclusion and further extensions}

Theorem~\ref{theorem:BLP-generalized} can be interpreted as indication that the existence of the left limit in~\eqref{eq:BLP-generalized} is not an exceptional feature of left-c.e.\ Martin-Löf random reals but is rather an inherent property of Solovay reducibility to arbitrary Martin-Löf random reals via nondecreasing translation functions.

This allows to suppose that the intuitive idea of Solovay reducibility from $\alpha$ to $\beta$, namely the existence of a \say{faster} approximation of $\alpha$ than a given approximation of $\beta$, can be captured in terms of monotone translation functions.
A characterization of that kind of the {S2a}-reducibility has been found in 2023 by Kumabe, Miyabe, and Suzuki~\cite[Theorem 3.7]{Kumabe-etal-2023}.

\begin{theorem}[Kumabe et al., 2023]\label{Kumabe-characteristic}
    Let $\alpha$ and $\beta$ be two c.a.\ reals. Then $\alpha\redsolovayzweia\beta$ if and only if there exist a lower semi-computable Lipschitz function ${f:\mathbb{R}\to\mathbb{R}}$ and an upper semi-computable Lipschitz function ${h: \mathbb{R}\to\mathbb{R}}$ such that $f(x)\leq h(x)$ for all $x\in\mathbb{R}$  and ${f(\beta) = h(\beta) = \alpha}$.
\end{theorem}

We conjecture that the Limit Theorem of Barmpalias and Lewis-Pye can be generalized on the set of c.a.\ reals for the {S2a}-reducibility.

\begin{conjecture}
    Let $\alpha$ be a c.a.\ real and $\beta$ be a Martin-Löf random c.a.\ real that fulfills $\alpha\redsolovayzweia\beta$ via functions $f$ and $h$ as in Theorem~\ref{Kumabe-characteristic}. Then there exists a constant $d$ such that
    \begin{equation}\label{eq:BLP-RZ}
        f'(\beta) = h'(\beta) = d,
    \end{equation}
    where $d$ does not depend on the choice of $f$ and $h$ witnessing the reducibility $\alpha\redsolovayzweia\beta$.
    Moreover, $d=0$ if and only if $\alpha$ is not Martin-Löf random.
\end{conjecture}

Finally, by Merkle and Titov~\cite[Corollary~2.10]{Merkle-Titov-2022}, the set of Schnorr random reals is closed upwards relative to the Solovay reducibility via total translation functions. We still don't know whether the Limit Theorem of Barmpalias and Lewis-Pye can be adapted for Schnorr randomness.

\section*{Acknowledgements}
The main result of this article, Theorem~\ref{theorem:BLP-generalized}, is a somewhat strengthened version of the main result of my doctoral dissertation.
I would like to thank my advisor Wolfgang Merkle for supervising the dissertation and for discussion that helped me to improve its presentation.


\begin{thebibliography}{1}

\bibitem{Barmpalias-Lewispye-2017}
George Barmpalias and Andrew Lewis-Pye.
Differences of halting probabilities. 
\emph{Journal of Computer and System Sciences}, 89:349--360 (2017)

\bibitem{Calude-etal-2001}
Christian Calude, Peter Hertling, Bakhadyr Khoussainov, and Yongge Wang.
Recursively enumerable reals and Chaitin $\Omega$ numbers. In M. Morvan, C. Meinel, and D. Krob, editors, 
\emph{STACS 98. 15th Annual Symposium on Theoretical Aspects of Computer Science 1998. Proceedings},
volume 1373 of Lecture Notes in Computer Science, pp 596--606. Springer, Berlin (1998)

%
\bibitem{Downey-Hirschfeldt-2010}
Rodney Downey and Denis Hirschfeldt.
\emph{Algorithmic Randomness and Complexity}. Springer, Berlin (2010)

\bibitem{Hoyrup-etal-2018}
Mathieu Hoyrup, Diego Nava Saucedo, Donald M Stull.
Semicomputable geometry.
\emph{ICALP 2018 -- 45th International Colloquium on Automata, Languages, and Programming},
Prague (2018)

\bibitem{Kucera-Slaman-2001}
Antonín Kučera and Theodore Slaman.
Randomness and recursive enumerability.
\emph{SIAM Journal on Computing}, 31:199--211 (2001)

\bibitem{Kumabe-etal-2020}
Masahiro Kumabe, Kenshi Miyabe, Yuki Mizusawa, and Toshio Suzuki.
Solovay reducibility and continuity.
\emph{Journal of Logic and Analysis}, 12(2):1--22 (2020)

\bibitem{Kumabe-etal-2023}
Masahiro  Kumabe, Kenshi Miyabe, and Toshio Suzuki.
Solovay reducibility via Lipschitz functions and signed-digit representation.
\emph{ Computability},
vol.\ Pre-press, no. Pre-press, pp 1--27 (2024)

%
\bibitem{Merkle-Titov-2022}
Wolfgang Merkle and Ivan Titov.
A total Solovay reducibility and totalizing of the notion of speedability.
\emph{13th Panhellenic Logic Symposium, University of Thessaly, Volos, Greece, July 6-10, 2022, Proceedings},
vol.\ II, pp 68--78 (2022)

\bibitem{Miller-2017}
Joseph Miller.
On work of Barmpalias and Lewis-Pye: A derivation on the d.c.e.\ reals.
\emph{Lecture Notes in Computer Science}, 10010:644--659 (2017)

%
%
%
\bibitem{Rettinger-Zheng-2021}
Robert Rettinger and Xizhong Zheng.
Computability of real numbers.
\emph{Handbook of Computability and Complexity in Analysis},
pp 3--28, Springer, Cham (2021)

\bibitem{Solovay-1975}
Robert Solovay.
Draft of paper (or series of papers) on Chaitin’s work. 
\emph{Unpublished notes}, 215 pages (1975)

\bibitem{Titov-2023}
Ivan Titov.
Solovay reducibility and speedability outside of left-c.e.\ reals.
\emph{Dissertation zur Erlangung der Doktorwürde der Gesamtfakultät für Mathematik, Ingenieur- und Naturwissenschaften der Ruprecht-Karls-Universität Heidelberg},
heiDOK (2024) 

\bibitem{Titov-2024-next}
Ivan Titov.
Solovay reducibility implies S2a-reducibility.
\emph{Unpublished technical report, to be submitted for publication}
(2024) 

\bibitem{Zheng-Rettinger-2004}
Xizhong Zheng and Robert Rettinger.
On the extensions of Solovay-reducibility.
\emph{COCOON 2004: Computing and Combinatorics},
pp 360--369 (2004)

\end{thebibliography}
\end{document}